\documentclass[12pt]{article}
\usepackage[dvips]{graphicx}
\usepackage{amsmath}
\usepackage{amsthm}
\usepackage{amsfonts}
\usepackage{amssymb}
\usepackage{setspace}
\usepackage{url}

\newtheorem{thm}{Theorem}[section]
\newtheorem{prop}[thm]{Proposition}
\newtheorem{cor}[thm]{Corollary}
\newtheorem{lem}[thm]{Lemma}

\theoremstyle{definition}
\newtheorem{defn}[thm]{Definition}

\newtheorem{rem}[thm]{Remark}
\newtheorem{eg}[thm]{Example}

\newcommand{\PC}{\mathcal{P}}
\newcommand{\basicinners}{\mathcal{I}}

\newcommand{\MCS}{\mathrm{MCS}(\Gamma)}

\newcommand{\QAutW}{{\operatorname{Aut}^1 W}}

\newcommand{\ab}{{\rm ab}}
\newcommand{\Ab}{{\rm Ab}}
\newcommand{\Aut}{\operatorname{Aut}}
\newcommand{\AutW}{\Aut W}
\newcommand{\CAutW}{\operatorname{Aut}^\ast W}
\newcommand{\COutW}{\operatorname{Out}^\ast W}
\newcommand{\AutZ}{\operatorname{{Aut}^0} W}
\newcommand{\Inn}{\operatorname{Inn}}
\newcommand{\InnW}{\Inn W}

\newcommand{\Inner}{\Inn W}
\newcommand{\Out}{\mathcal{O}ut}
\newcommand{\OutW}{\Out \, W}
\newcommand{\COutWQ}{\mathcal{O}ut^\ast \, W}

\newcommand{\OutZ}{\operatorname{Out}^0 W}
\newcommand{\OutZQ}{\mathcal{O}ut^0 \, W}

\newcommand{\PCT}{\mathcal{S}}

\newcommand{\I}{\mathcal{I}}

\newcommand{\PCMinus}{\PC^0}
\newcommand{\PartitionSet}{\mathcal{L}}
\newcommand{\M}{\mathcal{M}}

\newcommand{\U}{\mathcal{U}}

\newcommand{\Integer}{\mathbb {Z}}

\newcommand{\Nat}{\mathbb {N}}

\newcommand{\vcd}{\operatorname{vcd}}

\newcommand{\cyclicallyreduced}[1]{[#1]}
\newcommand{\retractGAutWToCAutW}{r}

\newcommand{\vertices}{V}
\newcommand{\edges}{E}
\newcommand{\leaves}{V_1}

\newcommand{\abs}[1]{\left\vert#1\right\vert}

\newcommand{\ordermap}{{\mathbf m}}

\newcommand{\Restriction}{\rho}

\newcommand{\supp}{\operatorname{supp}}

\begin{document}
\title{On the automorphisms of a graph product of abelian groups}
\author{Mauricio Gutierrez \and Adam Piggott \and Kim Ruane\\ \\Department of Mathematics, Tufts University\\ 503 Boston Ave, Medford MA 02155, USA\\ \\
Mauricio.Gutierrez@tufts.edu \\
Adam.Piggott@tufts.edu\\
Kim.Ruane@tufts.edu\\ \\
Corresponding Author: Adam Piggott}
\date{August 17, 2007}
\maketitle

\begin{abstract}
We study the automorphisms of a graph product of
finitely-generated abelian groups $W$.  More precisely, we study a natural subgroup $\CAutW$ of $\AutW$, with $\CAutW = \AutW$ whenever
vertex groups are finite and in a number of other cases.
We prove a number of structure results, including a semi-direct product decomposition
$\CAutW = (\InnW \rtimes \OutZ ) \rtimes \QAutW$.  We also give a number of applications, some of which are geometric in nature.
\end{abstract}



\section{Introduction}

The graph product of groups construction was first defined by Green
\cite{GreenThesis}. It interpolates between the free product
construction, in the case that $\Gamma$ is a discrete graph, and the direct
product construction, in the case that $\Gamma$ is a complete graph. The
class of graph products of finitely-generated abelian groups
contains a number of important subclasses that are often
treated separately. In the present article we pursue a
unified treatment of the automorphisms of such
groups. Our methods are combinatorial.  Our results have a number of applications which are
geometric in nature.

A non-trivial finite simplicial graph $\Gamma = \Gamma(\vertices,
\edges)$ is a pair consisting of a non-empty finite set $\vertices =
\{v_1, v_2, \dots, v_N\}$ (the \emph{vertices}) and a set $\edges$
(the \emph{edges}) of unordered pairs from $\vertices$.
We say that vertices $v_i, v_j$ are \emph{adjacent} if $\{v_i, v_j\} \in \edges$.  We consider
$\Gamma$ to be a metric object in the usual way, with $d_\Gamma$
denoting the distance function. An \emph{order map (on $\Gamma$)} is
a function
$$\ordermap\!: \{1, 2, \dots, N\} \to \{p^\alpha \; | \; p \hbox{ prime and }
\alpha \in \Nat\} \cup \{\infty\}.$$ A pair $(\Gamma, \ordermap)$ is called a \emph{labeled graph} and
determines a group $W(\Gamma, \ordermap)$ with the following presentation (by convention, the relation $v_i^{\infty}$ is the trivial relation):
\begin{equation}\label{PresentationOfW}
\langle \vertices \; | \; v_i^{\ordermap(i)},  v_j v_k v_j^{-1} v_k^{-1} \;\; (1 \leq i, j, k \leq N, j < k, d_\Gamma(v_j, v_k) = 1)\rangle.
\end{equation}
We say that $W(\Gamma, \ordermap)$ is a \emph{graph product of
directly-indecomposable cyclic groups}.  Following an established convention, we do not distinguish between a vertex of $\Gamma$ and
the corresponding generator of $W(\Gamma, \ordermap)$.

The class of graph products of directly-indecomposable cyclic groups is identical to the class of graph products of finitely-generated abelian groups
for the following reason: if $G$ is group and $G$ is isomorphic to a graph product of finitely-generated abelian groups, then there exists a unique isomorphism class of
labeled-graphs $(\Gamma, \ordermap)$ such that $G \cong W(\Gamma, \ordermap)$ \cite{MoAndAdamRigidity}.  Empowered by this fact, we usually omit
mention of $\Gamma$ and $\ordermap$ from the notation, writing $W := W(\Gamma, \ordermap)$.
The important subclasses alluded to in the opening paragraph include finitely-generated
abelian groups ($\Gamma$ a complete graph), graph products of primary cyclic groups ($\ordermap(i) < \infty$ for
each $i$), right-angled Coxeter groups
($\ordermap(i) = 2$ for each $i$) and right-angled Artin groups ($\ordermap(i) = \infty$ for each $i$).

For a full subgraph $\Delta$ of $\Gamma$, we write $W(\Delta)$ for
the subgroup (known as a \emph{special subgroup}) of $W$
generated by the vertices in $\Delta$.  We write $\MCS$ for the set
of maximal complete subgraphs (or cliques) of $\Gamma$. The subgroups
$W(\Delta)$, $\Delta \in \MCS$, will be called the \emph{maximal complete subgroups} of $W$.  If $W$ is a graph product of primary cyclic groups, then the maximal
complete subgroups are a set of representatives for the conjugacy classes of maximal finite subgroups of $W$ \cite[Lemma 4.5]{GreenThesis}
and each automorphism of $W$ maps each maximal complete subgroup to a conjugate of some maximal complete subgroup.  This is not true in an arbitrary graph product of
directly-indecomposable cyclic groups, but we may pretend that it is by restricting our attention to a natural subgroup of $\AutW$.

\begin{defn}
Write $\CAutW$ for the subgroup of $\AutW$ consisting of those automorphisms which map each maximal complete
subgroup to a conjugate of a maximal complete subgroup.
\end{defn}

The following lemma is immediate from the discussion above and the main result of \cite{Laurence}.
For each $1 \leq i \leq N$, we write $L_i$ (resp, $S_i$) for the \emph{link} (resp. \emph{star}) of $v_i$.

\begin{lem}\label{WhenGAutWEqualsAutW}
If $W$ is a graph product of directly-indecomposable
cyclic groups, then $\CAutW = \AutW$ in each of the following cases:
\begin{enumerate}
\item \label{GPPCGCondition} $W$ is a graph product of primary cyclic groups;
\item \label{NoNGTCondition} $W$ is a right-angled Artin group and $L_i \not \subseteq L_j$ for each pair of distinct non-adjacent vertices $v_i, v_j \in \vertices$;
\item \label{EasyNoNGTCondition} $W$ is a right-angled Artin group and $\Gamma$ contains no vertices of valence less than two and no circuits of length less than 5.
\end{enumerate}
\end{lem}

\begin{rem}
Case (\ref{NoNGTCondition}) can be substantially generalized to groups that are not right-angled Artin groups.
\end{rem}

We now report the main results of the present article.  They concern the structure of $\CAutW$ and shall make
reference to the subgroups and quotients of $\Aut W$ defined in Figure \ref{SubgroupDefinitionTable}.
In writing $\AutZ$ for the subgroup of
`conjugating automorphisms', we follow Tits \cite{Tits}.
M{\"u}hlherr \cite{Bernie} writes ${\rm Spe}(W)$ for the same subgroup.  Charney, Crisp and Vogtmann \cite{CharneyCrispVogtmann} use the
notation $\AutZ$ and $\OutZ$ for different subgroups of the automorphism group of a right-angled Artin group than described here.

\begin{figure}
\begin{tabular}{|l|l|}
  \hline
  Group & Description \\
  \hline
  & \\ $\CAutW$ & Those automorphisms of $W$ which map each maximal complete\\
           &  subgroup to a conjugate of a maximal complete subgroup\\ \hline
  & \\ $\QAutW$ & Those automorphisms of $W$ which map each maximal complete\\
           & subgroup to a maximal complete subgroup\\ \hline
  & \\ $\AutZ$  & Those automorphisms of $W$ which map each vertex $v_i \in \vertices$\\
           & to a conjugate of itself\\ \hline
  & \\ $\InnW$  & The inner automorphisms of $W$\\ & \\ \hline
  & \\ $\COutW$  & The subgroup of $\AutW$ generated by the set $\PCMinus \sqcup \QAutW$ \\ &  (see Definition \ref{DefnOfQ})    \\ \hline
  & \\ $\OutZ$  & The subgroup of $\AutW$ generated by the set $\PCMinus$ \\ &\\ \hline
  & \\ $\OutW$  & The quotient $\AutW / \InnW$\\ & \\ \hline
  & \\ $\COutWQ$  & The quotient $\CAutW / \InnW$\\ & \\ \hline
  & \\ $\OutZQ$  & The quotient $\AutZ / \InnW$\\ & \\ \hline
\end{tabular}
\caption{Subgroups and quotients of $\AutW$.
\label{SubgroupDefinitionTable}}
\end{figure}

Tits \cite{Tits} proved that if $W$ is a right-angled Coxeter group, then $\AutW  = \AutZ \rtimes \QAutW$.
Our first main result is a generalization of Tits' splitting.

\begin{thm}[cf. \cite{Tits}]\label{TitsSplittingRecovered}
If $W$ is a graph product of directly-indecomposable
cyclic groups, then $$\CAutW = \AutZ \rtimes \QAutW.$$
\end{thm}

If $W_{\ab}$ denotes the abelianization of $W$, then the subgroup $\QAutW$ is isomorphic to the image
of $\CAutW$ under the natural map $\AutW \to \AutW_{\ab}$. In particular, $\QAutW$ is finite
in the case that $W$ is a graph product of primary cyclic groups.

We next turn our attention to the study of $\AutZ$.  Choosing the right generating set for various subgroups of $\AutZ$ shall be important throughout.
For $1 \leq i \leq N$ and $K$ a (non-trivial) connected
component of $\Gamma \setminus S_i$, we write $\chi_{iK}$ for the
automorphism of $W$ determined by
$$\chi_{iK}(v_j) = \left\{%
\begin{array}{ll}
  v_iv_jv_i^{-1} & \hbox{if } v_j \in K, \\
  \;\,\, \,\, v_j & \hbox{if } v_j \not \in K. \\
\end{array}%
\right.$$  Such an automorphism is called a \emph{partial conjugation} with \emph{operating letter} $v_i$ and \emph{domain} $K$.  We write $\PC$ for the set of partial conjugations
(see $\S$\ref{ExampleSection} for an example).
Laurence \cite[Theorem 4.1]{LaurenceThesis} proved that $\AutZ$ is generated by $\PC$.

For a subgraph $\Omega \subseteq \Gamma$, we write $pr_\Omega$ for the retraction map $W \to W(\Omega)$ and
$\PC_\Omega := \{\chi_{i Q} \in \PC \; | \; v_i \in \Omega\}$.
For $\phi \in \AutZ$
and $w_1, \dots, w_N \in W$ such that $\phi(v_i) = w_i v_i w_i^{-1}$ for each $1 \leq i \leq N$, we write
$\phi_\Omega$ for the map $V \to W$ defined by
$$v_i \mapsto pr_\Omega(w_i) .v_i .pr_\Omega(w_i)^{-1} \text{ for each } 1 \leq i \leq N.$$
We shall show that $\phi_\Omega$ extends to an automorphism of $W$, also denoted by $\phi_\Omega$.  In fact, the following holds:

\begin{thm}\label{RetractionRewritingLemma}
For each subgraph $\Omega \subseteq \Gamma$, the map $\phi \mapsto \phi_\Omega$ is a
retraction homomorphism $\AutZ \to \langle \PC_\Omega \rangle$.
\end{thm}

The following immediate corollary is key to a number of our arguments.

\begin{cor}[The Restricted Alphabet Rewriting Lemma]\label{RestrictedAlphabetRewriting}
If $\phi \in \AutZ$ and there
exist $z_1, \dots, z_N \in W(\Omega)$ such that $\phi(v_j) = z_j v_j z_j^{-1}$ for each $1 \leq j \leq N$, then
any word for $\phi$ in the alphabet
$\PC^{\pm 1}$ may be rewritten as a word in
the alphabet $\PC_\Omega^{\pm 1}$ (a word
which still spells $\phi$) by simply omitting those generators not
in $\PC_\Omega^{\pm 1}$.
\end{cor}

Next we define a subset $\PCMinus \subset \PC$ (see Definition
\ref{DefnOfQ}) such that $\AutZ$ is generated by the disjoint union of
$\PCMinus$ and the inner automorphisms. Using The Restricted Alphabet Rewriting Lemma we show that $\InnW \cap \langle \PCMinus \rangle = \{id\}$ and hence prove our next main result.

\begin{thm}\label{CSemiDirect}
If $W$ is a graph product of directly-indecomposable
cyclic groups, then $$\AutZ = \Inner \rtimes \OutZ.$$  Consequently, $\OutZ \cong \OutZQ$ and $\COutW \cong \COutWQ$.
\end{thm}

The subgroup $\InnW$ is well-understood and is itself a special subgroup of $W$ (see Lemma \ref{CenterOfW})

In $\S$\ref{ApplicationsSection} we give sufficient conditions for the splittings of Theorems \ref{TitsSplittingRecovered} and \ref{CSemiDirect}
to be compatible (Lemma \ref{WhenAreSplittingsCompatibleLemma}) and then
record an application to the theory of group extensions (Corollary \ref{ExtensionApplication} and Corollary \ref{RightAngledExtension}).

To each partial conjugation we associate a subset of $\vertices$ called the set of \emph{link points} of the partial conjugation
(see Definition \ref{LinkPointDefn}).  If $\Gamma$ is connected, then elements of $\PCMinus$ which do not share a common link point must commute
(see Lemma \ref{GeometryOfPCsLemma} and Remark \ref{Case13ImpossibleRemark}).  So, in anticipation of
studying $\OutZ$, it is natural to consider sets of partial conjugations which share a common link point.  We prove the following:

\begin{thm}\label{QiEmbedsInAutZ}
If $W$ is a graph product of directly-indecomposable
cyclic groups and $1 \leq i \leq N$ and $\PartitionSet_i \subset \PC$ is the
set of partial conjugations for which $v_i$ is a link point, then the natural restriction homomorphism
$\Restriction_i: \langle \PartitionSet_i \rangle \hookrightarrow \AutZ(L_i)$ is injective.
\end{thm}

We now consider $\OutZ$.  We give a simple graph criterion which characterizes when $\OutZ$ is abelian.

\begin{defn}[SIL]\label{SILDefn}
We say that $\Gamma$ contains a \emph{separating intersection of links} (SIL) if there exist $1 \leq i < j \leq N$ such that
the following conditions hold:
\begin{enumerate}
\item $d(v_i, v_j) \geq 2$;
\item there exists a connected component $R$ of $\Gamma \setminus (L_i \cap L_j)$ such that $v_i, v_j \not \in R$.
\end{enumerate}
\end{defn}

\begin{thm}\label{OutZAbelianIffSIL}
If $W$ is a graph product of directly-indecomposable
cyclic groups, then the following are equivalent:
\begin{enumerate}
\item $\OutZ$ is an abelian group;
\item $\Gamma$ does not contain a SIL.
\end{enumerate}
\end{thm}

In the case that $W$ is a graph product of primary cyclic groups, this
also characterizes when $\OutW$ is finite.

\begin{cor}\label{OutWInfiniteIffSIL}
If $W$ is a graph product of primary cyclic groups, then the following are equivalent:
\begin{enumerate}
\item $\OutZ$ is an abelian group;
\item $\Gamma$ does not contain a SIL;
\item $\OutW$ is finite.
\end{enumerate}
\end{cor}

In $\S$\ref{ApplicationsSection} we describe a number of applications of Corollary \ref{OutWInfiniteIffSIL} to the study of the geometry of $W$ and $\AutW$.
We examine the combined effect of the existence (or absence) of an SIL and certain
other graph properties which are known to determine geometric properties of $W$ such as word-hyperbolicity (Corollary \ref{OutWInfiniteMeansNotHyperbolic}),
the isolated flats property
(Corollary \ref{IsolatedFlatsAndOutWInfinite}) and whether or not $W$ can act on a CAT(0) space with locally-connected visual boundary
(Corollary \ref{OutWInfiniteMeansNLC}).  We also characterize when
$\AutW$ is word hyperbolic (Corollary \ref{AutWHyperbolic}).

In the special case that $\Gamma$ is a tree, our study of $\OutZ$ can proceed much further.  In this case, each partial conjugation has a unique link
point, the subsets $\PartitionSet_i^0 := \PartitionSet_i \cap \PCMinus$
partition $\PCMinus$ and the partition corresponds to a direct product decomposition of $\OutZ$.  Since each
$W(L_i)$ is a free product of cyclic groups and each image $\Restriction_i(\langle \PartitionSet_i^0 \rangle)$
is easily understood (see Proposition \ref{QiInCaseGammaATree}), we are able to give a complete description of $\OutZ$.
\begin{thm}\label{NiceStructureOfOutZForTree}
If $W$ is a graph product of directly-indecomposable
cyclic groups and $\Gamma$ is a tree, then
$$\OutZ \cong \Ab \times \Bigl(\prod_{i=1}^N \OutZ(L_i)\Bigr)$$ for a finitely-generated abelian
group $\Ab$ as described in Remark \ref{PresentingCAutWRemark}. In particular:
\begin{enumerate}
\item if $W$ is a right-angled Artin group, then $\Ab$ is
a free abelian group;
\item if $W$ is a graph product of primary cyclic groups, then $\Ab$ is a finite abelian group.
\end{enumerate}
\end{thm}

We apply Theorem \ref{NiceStructureOfOutZForTree} to determine a finite presentation for $\CAutW$ (Remark \ref{PresentingCAutWRemark}),
to calculate the virtual cohomological dimension of $\OutW$ in the case that $W$ is a graph product of primary cyclic groups (Corollary \ref{OutWVTFandFVCD}) and to
prove the existence of regular languages of normal forms for $\OutZ$ and $\AutZ$ in the case that $W$ is a right-angled Artin group (Corollary \ref{RLofNF}).

We mention that Laurence \cite{LaurenceThesis} and M{\"u}hlherr \cite{Bernie} independently determined finite presentations for $\AutZ$
in the case that $W$ is a right-angled Coxeter group. Castella \cite{Castella} determined a finite presentation of $\QAutW$ for a certain subclass of
right-angled Coxeter groups.  Proving a conjecture of Servatius \cite{Servatius}, Laurence \cite{Laurence}
determined a generating set for $\AutW$ in the case that $W$ is a right-angled Artin
group---a generating set for $\CAutW$ can be deduced from this list.  To the best of the authors' knowledge, Laurence's unpublished Ph.D.
Thesis \cite{LaurenceThesis} is the only previous work to consider
the automorphisms of graph products of abelian groups in a unified way.

We now briefly describe the structure of the present article.  We attend to some preliminaries in $\S$\ref{Preliminaries}.  Theorem
\ref{TitsSplittingRecovered} is the topic of $\S$\ref{TitsSplittingSection}.   Theorem \ref{RetractionRewritingLemma} is proved in $\S$\ref{RestrictedAlphabetRewritingSection}.
The Restricted Alphabet Rewriting Lemma is then used to prove Theorem \ref{CSemiDirect} in $\S$\ref{SplittingAutZSection}.
Theorem \ref{QiEmbedsInAutZ} is proved in $\S$\ref{ConnectedSubSection}.  Theorem \ref{OutZAbelianIffSIL} and Corollary \ref{OutWInfiniteIffSIL}
are proved in $\S$\ref{CardinalityOfOutZ}.
Theorem \ref{NiceStructureOfOutZForTree} is proved
in $\S$\ref{TreeSection}.  We describe a number of applications of our results in $\S$\ref{ApplicationsSection}.  In $\S$\ref{ExampleSection} we illustrate a number of definitions and results by following an example.
Each section is prefaced by a short description of its contents.

\section{Preliminaries}\label{Preliminaries}

In this section we establish notation and remind the reader of some
fundamental results concerning graph products of groups.  We have stated the results only for the class of graph products of directly-indecomposable
cyclic groups.  They appear, in more general form, in Elisabeth Green's Ph.D. Thesis \cite{GreenThesis}
and Michael Laurence's Ph.D. Thesis \cite{LaurenceThesis}, proved by different methods.

By a subgraph of $\Gamma$ we shall always mean a
full subgraph.  Thus a subgraph $\Delta = (\vertices_\Delta,
\edges_\Delta)$ is determined by a subset $\vertices_\Delta
\subseteq \vertices$ and the rule
$\edges_\Delta = \{\{v_i, v_j\} \in \edges \; | \; v_i, v_j \in \vertices_\Delta\}.$

We remind the reader that, for each $1 \leq i \leq N$, we write $L_i$ (resp, $S_i$) for the \emph{link} (resp. \emph{star}) of $v_i$.  That is,
$L_i$ (resp. $S_i$) is the subgraph of $\Gamma$ generated by the vertices $\{v_j \in \vertices \; | \; d(v_i, v_j) = 1\}$ (resp.
$\{v_j \in \vertices \; | \; d(v_i, v_j) \leq 1\}$).

\begin{defn}
Let $1 \leq i_1, \dots, i_k \leq N$ be such that $i_{j} \neq i_{j+1}$ and let $\alpha_1, \dots, \alpha_k$ be non-zero integers.
Each $v_{i_j}^{\alpha_j}$ is a \emph{syllable} of the word
$v_{i_1}^{\alpha_1} v_{i_2}^{\alpha_2} \dots v_{i_k}^{\alpha_k},$
and we say that the word is \emph{reduced} if there is no word with fewer syllables which spells the same element of $W$.  We say that consecutive syllables
$v_{i_j}^{\alpha_j}$, $v_{i_{j+1}}^{\alpha_{j+1}}$ are \emph{adjacent} if $v_{i_j}$ and $v_{i_{j+1}}$ are.
\end{defn}

\begin{lem}[The Deletion Condition]\label{DeletionCondition}
Let  $1 \leq i_1, \dots, i_k \leq N$ be such that $i_{j} \neq i_{j+1}$ and let $\alpha_1, \dots, \alpha_k$ be non-zero integers.
If the word $v_{i_1}^{\alpha_1} v_{i_2}^{\alpha_2} \dots v_{i_k}^{\alpha_k}$ is not reduced, then there exist $p, q$ such that $1 \leq p < q \leq k$,
$v_{i_p} = v_{i_q}$ and $v_{i_p}$ is adjacent to each vertex $v_{i_{p+1}}, v_{i_{p+2}}, \dots, v_{i_{q-1}}$.
\end{lem}

\begin{lem}[Normal Form]\label{EquivalenceOfReducedWords}
Let  $1 \leq i_1, \dots, i_k, j_1, \dots, j_k \leq N$ and
$\alpha_1, \dots,$ $\alpha_k, \beta_1, \beta_2, \dots, \beta_k \in \Integer \setminus \{0\}$.  If
$v_{i_1}^{\alpha_1} v_{i_2}^{\alpha_2} \dots v_{i_k}^{\alpha_k}$
and
$w_{j_1}^{\beta_1} w_{j_2}^{\beta_2} \dots w_{j_k}^{\beta_k}$
are reduced words which spell the same element of $W$, then the first word may be transformed into the second by repeatedly
swapping the order of adjacent syllables.
\end{lem}

The following three results witness the importance of special subgroups to the study of $W$.

\begin{lem}\label{SpecialSubgroupsEmbed}
Let $\Delta$ be a subgraph of $\Gamma$.  The natural map $W(\Delta) \to W(\Gamma)$ is an embedding.
\end{lem}

\begin{lem}\label{CenterOfW}
The center of $W$ is the special subgroup generated by the vertices
$\{v_i \in \vertices \; | \; d_\Gamma(v_i, v_j) \leq 1 \hbox{ for each } 1 \leq j \leq N\}.$
Further, the center of $W$ is finite in the case that $W$ is a graph product of primary cyclic groups.
\end{lem}

It follows that $\InnW$ is isomorphic to the special subgroup generated by the vertices
$\{v_i \in \vertices \; | \; d_\Gamma(v_i, v_j) > 1 \hbox{ for some } 1 \leq j \leq N\}$.  Further,
$\InnW$ is isomorphic to a finite-index subgroup of $W$ in the case that $W$ is a graph product of primary cyclic groups.

\begin{lem}
A special subgroup $W(\Delta)$ has finite order if and only if $\Delta$ is
a complete graph and each vertex of $\Delta$ has finite order.  Further,
if a subgroup $H$ of $W$ has finite order, then $H$ is contained in some conjugate of a special subgroup of finite order.
\end{lem}

The centralizer of a vertex is easily understood.

\begin{lem}\label{CentralizerOfAVertex}
For each $1 \leq j \leq N$, the centralizer of $v_j$ in $W$ is the special subgroup generated by $S_j$.
\end{lem}

\section{A splitting of $\CAutW$}\label{TitsSplittingSection}

In this section we prove that $\CAutW = \AutZ \rtimes \QAutW$
(Theorem \ref{TitsSplittingRecovered}), thus generalizing a result
of Tits \cite{Tits}.  We shall prove
the result by exhibiting a retraction homomorphism $\CAutW \to \QAutW$ with kernel $\AutZ$.

For each $\Delta \in \MCS$ and $w \in W$ and $u \in w W(\Delta)
w^{-1}$, there exists a unique element $\cyclicallyreduced{u}$ of
minimal length in the conjugacy class of $u$. Equivalently, $\cyclicallyreduced{u}$ is the unique
element of $W(\Delta)$ in the conjugacy class of $u$.  For each
automorphism $\gamma \in \CAutW$, define $r(\gamma): \vertices \to
W$ by the rule $v \mapsto \cyclicallyreduced{\gamma(v)}.$

\begin{lem}\label{GAutToCAutLemmaOne}
For each automorphism $\gamma \in \CAutW$, the map
$r(\gamma)\!: \vertices \to W$ extends to an endomorphism of $W$.
\end{lem}

\begin{proof}
It suffices to show that the relations used to define $W$ are
`preserved' by $r(\gamma)$.

Let $1 \leq i \leq N$ be such that $\ordermap(i) < \infty$. Since the order of an element is preserved under
automorphisms and conjugation, $\bigl(r(\gamma)(v_i)\bigr)^{\ordermap(v_i)} = 1$ and $r(\gamma)$ preserves the relation $v_i^{\ordermap(v_i)} = 1$.

Let $1 \leq j < k \leq N$ be such that $d(v_j, v_k) = 1$. There exist $w
\in W$ and $\Delta \in \MCS$ and $a, b \in W(\Delta)$ such that
$\gamma(v_j) = w a w^{-1}$ and $\gamma(v_k) = w b w^{-1}$. Recall that $W(\Delta)$ is an abelian group.  Then
$$\bigl(r(\gamma)(v_j)\bigr) \bigl(r(\gamma)(v_k)\bigr) \bigl(r(\gamma)(v_j)\bigr)^{-1} \bigl(r(\gamma)(v_k)\bigr)^{-1} =  aba^{-1}b^{-1} = 1.$$
 Thus $r(\gamma)$ preserves the relation $v_j v_k v_j^{-1}v_k^{-1} = 1$.
\end{proof}

We abuse notation by writing $r(\gamma): W \to W$ for the endomorphism of $W$ determined by $r(\gamma)\!: \vertices \to W$.

\begin{lem}\label{CommutingProductsEasyToUnderstandLemma}
For each automorphism $\delta \in \CAutW$ and $\Delta \in \MCS$ and $a \in W(\Delta)$, we have $r(\delta)(a) = [\delta(a)]$.
\end{lem}

\begin{proof}
We have $a = d_1^{\epsilon_1} \dots d_q^{\epsilon_q}$ for some vertices  $d_1, \dots, d_q$ in $\Delta$ and some integers $\epsilon_1, \dots, \epsilon_q$. By the definition of $\CAutW$, there
exist $w \in W$ and $\Theta \in \MCS$ such that $\delta(W(\Delta)) = wW(\Theta)w^{-1}$.
Hence there exist $t_1, \dots, t_q \in W(\Theta)$ such that $\delta(d_i) = w t_i w^{-1}$ for each $1 \leq i \leq q$.
Then
\begin{multline*}
r(\delta)(a) = r(\delta)(d_1)^{\epsilon_1} \dots r(\delta)(d_q)^{\epsilon_q} = t_1^{\epsilon_1} \dots t_q^{\epsilon_q}  = [w t_1^{\epsilon_1} w^{-1} \dots w t_q^{\epsilon_q} w^{-1}] = [\delta(a)],
\end{multline*}
as required.
\end{proof}

\begin{lem}\label{GAutToCAutLemmaTwo}
For each pair of automorphisms $\gamma, \delta \in
\CAutW$, we have $r(\delta \gamma) = r(\delta) r(\gamma)$.
\end{lem}

\begin{proof}
Let $\gamma, \delta \in \CAutW$ and let $1 \leq i \leq N$.
There exist $\Delta \in \MCS$ and $a \in
W(\Delta)$ and $w_1 \in W$ such that $\gamma(v_i) = w_1 a w_1^{-1}$.  There exist $\Theta \in \MCS$, $b \in W(\Theta)$ and $w_2 \in W$ such that
$\delta(W(\Delta)) = w_2 W(\Theta) w_2^{-1}$ and $\delta(a) = w_2 b w_2^{-1}$.  By Lemma \ref{CommutingProductsEasyToUnderstandLemma} we have
that $r(\delta)(a) = b$.
Then
\begin{multline*}
r(\delta \gamma)(v_i) = [\delta \gamma(v_i)] = [\delta(w_1) w_2 b w_2^{-1} \delta(w_1)^{-1}]  = b = r(\delta)(a) = r(\delta) r(\gamma)(v_i),
\end{multline*}
as required.
\end{proof}

\begin{lem}\label{ImageOfRetractionLemma}
For each $\gamma \in \CAutW$, $r(\gamma) \in \QAutW$.
\end{lem}

\begin{proof}
Let $\gamma \in \CAutW$.  By Lemma \ref{GAutToCAutLemmaTwo}, $r(\gamma^{-1})r(\gamma) = r(\gamma^{-1} \circ \gamma) = r(id) = id$ and
$r(\gamma)$ is an automorphism of $W$.  It is clear from the definitions that $r(\gamma) \in \QAutW$.
\end{proof}

\begin{prop}\label{RetractionThm}
The map $r$ is a retraction homomorphism $\CAutW \to \QAutW$ with kernel $\AutZ$.
\end{prop}

\begin{proof}
By Lemmas \ref{GAutToCAutLemmaTwo} and \ref{ImageOfRetractionLemma}, $r$ is a homomorphism $\CAutW \to \QAutW$.  It is clear from the definitions that
$\retractGAutWToCAutW$ restricts to the identity map on $\QAutW$ and $\retractGAutWToCAutW$ has kernel $\AutZ$.
\end{proof}

\begin{rem}[Tits' approach and Theorem \ref{TitsSplittingRecovered}]
For each $\Delta \in \MCS$, we may consider $W(\Delta)$ as a subgroup of $W_{\ab}$.  Then the union
$$G = \underset{\Delta \in \MCS}{\bigcup} W(\Delta) \subset W_{\ab}$$
is a groupoid in the usual way.  In case $W$ is a right-angled Coxeter group, Tits \cite{Tits} identifies $\QAutW$ with the
groupoid automorphisms $\Aut G$ of $G$ and constructs a section of the obvious homomorphism $\CAutW \to \Aut G$.  This identification
carries over in the case that $W$ is an arbitrary graph product of directly-indecomposable groups and a section of the homomorphism $\CAutW \to \Aut G$
is defined similarly.
\end{rem}

\section{The group $\AutZ$}\label{InnOutSplittingSection}

The primary goal of this section is to prove that $\AutZ = \InnW \rtimes \OutZ$ (Theorem \ref{CSemiDirect}).
Before achieving this in $\S$\ref{SplittingAutZSection}, we prove The Restricted Alphabet Rewriting Lemma
in $\S$\ref{RestrictedAlphabetRewritingSection}.

We remind the reader that $\AutZ$ is generated by the set of partial conjugations $\PC$.




\begin{rem}\label{WLOGAssumeConnected}
Although we shall not assume that $\Gamma$ is connected throughout, it is occasionally convenient to note that such as assumption
places no restriction on our study of $\AutZ$.
For suppose that $\Gamma$ is not connected.  Write $(\Gamma^+, \ordermap^+)$ for the labeled-graph obtained from $(\Gamma, \ordermap)$ as follows:
\begin{enumerate}
\item introduce a new vertex $v_0$ and extend $\ordermap$ to a function $\ordermap^+$ with domain $\{0, 1, 2, \dots, N\}$
by making a choice $$\ordermap(0) \in \{p^\alpha \; | \; p \hbox{ prime and }
\alpha \in \Nat\} \cup \{\infty\};$$
\item add an edge from $v_0$ to $v_i$ for each $1 \leq i \leq N$.
\end{enumerate}
Write $W^+ := W(\Gamma^+, \ordermap^+)$ and write $S_i^+$ for the star of $v_i$ in $\Gamma^+$. For each $1 \leq i \leq N$, the connected
components of $\Gamma \setminus S_i$ are identical to the connected
components of $\Gamma^+ \setminus S_i^+$.  The subgraph
$\Gamma^+ \setminus S_0^+$ is empty.  Since $\PC$ generates $\AutZ$, it follows that $\AutZ \cong
\Aut^0 W^+$.
\end{rem}

\subsection{The Restricted Alphabet Rewriting Lemma}\label{RestrictedAlphabetRewritingSection}

In this subsection we prove Theorem \ref{RetractionRewritingLemma}.  As an aside, we illustrate the power of The Restricted Alphabet Rewriting Lemma
by proving that $\PC$ is a minimal generating set for $\AutZ$.

Throughout, we fix a subgraph $\Omega \subseteq \Gamma$.  Recall the definitions of $pr_\Omega$, $\PC_\Omega$ and $\phi_\Omega$ given in the
Introduction.

\begin{lem}\label{PhiUpsilonIsWellDefined}
The map $\phi \mapsto \phi_\Omega$ is well-defined
\end{lem}

\begin{proof}
Let $\phi \in \AutZ$ and $w_1, \dots, w_N \in W$ and $u_1, \dots, u_N \in W$ be such that
$\phi(v_i) = w_i v_i w_i^{-1} = u_i v_i u_i^{-1}$ for each $1 \leq i \leq N$.
Fix $1 \leq i \leq N$.  We must show that $pr_\Omega(w_i) . v_i . pr_\Omega(w_i)^{-1} = pr_\Omega(u_i) . v_i . pr_\Omega(u_i)^{-1}$.
Since  $w_i v_i w_i^{-1} = u_i v_i u_i^{-1}$, we have that $w_i^{-1} u_i$ is in the centralizer of $v_i$.
Recall that the centralizer of $v_i$ is generated by $S_i$ (Lemma \ref{CentralizerOfAVertex}).
Thus there exists $z_i \in \langle S_i \rangle$ such that
$u_i = w_i z_i$.  Since $pr_\Omega(S_i) \subseteq S_i \cup \{id\}$,
$pr_\Omega(z_i) \in \langle S_i \rangle$ and we have
\begin{eqnarray*}
pr_\Omega(w_i). v_i .pr_\Omega(w_i)^{-1} & = & pr_\Omega(u_i z_i). v_i .pr_\Omega(u_i z_i)^{-1} \\
 & = & pr_\Omega(u_i) .pr_\Omega(z_i) .v_i .pr_\Omega(z_i)^{-1} .pr_\Omega(u_i)^{-1} \\
  & = & pr_\Omega(u_i) .v_i .pr_\Omega(u_i)^{-1}.
\end{eqnarray*}
\end{proof}

\begin{lem}\label{PhiUpsilonIsAHomomorphism}
For each $\phi \in \AutZ$, the map $\phi_\Omega\!: \vertices \to W$ extends to a homomorphism $\phi_\Omega\!: W \to W$.
\end{lem}

\begin{proof}
Let $\phi \in \AutZ$ and $w_1, \dots, w_N \in W$ be such that
$\phi(v_i) = w_i v_i w_i^{-1}$ for each $1 \leq i \leq N$.  We must show that the map $\phi_\Omega$ `preserves' the defining relations
of $W$.

Let $1 \leq i \leq N$. Since $\phi_\Omega(v_i)$ is conjugate to $v_i$,  it has the same order as $v_i$ and the relation
$v_i^{\ordermap(i)}$ is preserved.

Let $1 \leq i < j \leq N$ be such that $v_i$ and $v_j$ are adjacent.  Since $\phi \in \AutZ$, it follows that there exists $w \in W$
such that $\phi(v_i) = w v_i w^{-1}$ and $\phi(v_j) = w v_j w^{-1}$.  Then $w_i = w z_i$ for some $z_i \in \langle S_i \rangle$ and
$w_j = w z_j$ for some $z_j \in \langle S_j \rangle$. So $w_i^{-1} w_j = z_i^{-1} z_j$ and
$pr_\Omega(w_i^{-1}).pr_\Omega(w_j) = pr_\Omega(z_i^{-1}).pr_\Omega(z_j).$
We have
\begin{eqnarray*}
& & \phi_\Omega(v_i).\phi_\Omega(v_j).\phi_\Omega(v_i)^{-1}.\phi_\Omega(v_i)^{-1} \\
& = & pr_\Omega(w_i). v_i .pr_\Omega(w_i^{-1}). pr_\Omega(w_j). v_j .pr_\Omega(w_j)^{-1} \\
&   & \;\;\;\;\;\;\;\;\;\;\;\;\;\;\;\;\;\; .pr_\Omega(w_i). v_i^{-1} .pr_\Omega(w_i)^{-1} .pr_\Omega(w_j) .v_j^{-1} .pr_\Omega(w_j)^{-1}\\
& = & pr_\Omega(w_i) .v_i .pr_\Omega(z_i^{-1}) .pr_\Omega(z_j). v_j .pr_\Omega(z_j)^{-1} \\
&   & \;\;\;\;\;\;\;\;\;\;\;\;\;\;\;\;\;\; .pr_\Omega(z_i) .v_i^{-1} .pr_\Omega(z_i)^{-1} .pr_\Omega(z_j). v_j^{-1}. pr_\Omega(w_j)^{-1}\\
& = & pr_\Omega(w_i) .pr_\Omega(z_i^{-1}) .v_i  v_j  v_i^{-1} v_j^{-1} .pr_\Omega(z_j) .pr_\Omega(w_j)^{-1}\\
& = & pr_\Omega(w_i z_i^{-1}) .1 .pr_\Omega(w_j z_j^{-1})^{-1}\\
& = & pr_\Omega(w). pr_\Omega(w)^{-1}\\
& = & 1,
\end{eqnarray*}
and the relation $v_i v_j v_i^{-1} v_j^{-1}$ is preserved.
\end{proof}

\begin{lem}\label{MUpsilonIsAHomomorphism}
For $\phi, \theta \in \AutZ$, $(\phi \circ \theta)_\Omega = \phi_\Omega \circ \theta_\Omega$.
\end{lem}

\begin{proof}
Fix $1 \leq i \leq N$. It suffices to show that $(\theta \circ \phi )_\Omega (v_i) = \theta_\Omega \circ \phi_\Omega(v_i)$.
Write $\vertices^\ast$ for the set of (not necessarily reduced) words in the alphabet $\vertices^{\pm 1}$.
Let $\mathcal{W}_1, \dots, \mathcal{W}_N \in \vertices^\ast$ and $\mathcal{U}_1, \dots, \mathcal{U}_N \in \vertices^\ast$
be such that $\phi(v_j) = \mathcal{W}_j v_j \mathcal{W}_j^{-1}$ and $\theta(v_j) = \mathcal{U}_j v_j \mathcal{U}_j^{-1}$
for each $1 \leq j \leq N$.  Let $\mathcal{T}_i$ be the word constructed from $\mathcal{W}_i$ as follows:
\begin{enumerate}
\item for each $1 \leq j \leq N$ and $\delta = \pm 1$, replace each occurrence of $v_j^\delta$ by $\mathcal{U}_j v_j^\delta \mathcal{U}_j^{-1}$;
\item append the word $\mathcal{U}_i$ to the resulting word;
\item omit those letters not in $\Omega$ from the resulting word.
\end{enumerate}
It is clear that $\mathcal{T}_i = pr_\Omega (\theta(\mathcal{W}_i) \mathcal{U}_i)$ and hence
$(\theta \phi)_\Omega (v_i) = \mathcal{T}_i v_i \mathcal{T}_i^{-1}$ (with equality in $W$).
Observe the following in the construction of $\mathcal{T}_i$:
\begin{enumerate}
\item [(OB1)] if $v_j \not \in \Omega$, then each occurrence of $v_j^\delta$ in $\mathcal{W}_j$ is eventually replaced by the word
$pr_\Omega(U_j) . pr_\Omega(U_j)^{-1}$, which is, of course, trivial in $W$;
\item [(OB2)] if $v_j \in \Omega$, then each occurrence of $v_j^\delta$ in $\mathcal{W}_j$ is eventually replaced by the word
$pr_\Omega(U_j) . v_j^\delta . pr_\Omega(U_j)^{-1}$.
\end{enumerate}
Let $\mathcal{T}'_i$ be the word constructed from $\mathcal{W}_i$ as follows:
\begin{enumerate}
\item for each $1 \leq j \leq N$ such that $v_j \not \in \Omega$ and each $\delta = \pm 1$, omit each occurrence of $v_j^\delta$;
\item for each $1 \leq j \leq N$ such that $v_j \in \Omega$ and each $\delta = \pm 1$,
replace each occurrence of $v_j^\delta$ by $pr_\Omega(\mathcal{U}_j). v_j^\delta .pr_\Omega(\mathcal{U}_j)^{-1}$;
\item append the word $pr_\Omega(\mathcal{U}_i)$ to the resulting word.
\end{enumerate}
It follows from (OB1) and (OB2) that $\mathcal{T}'_i = \mathcal{T}_i$ (with equality in $W$).
It is clear from the construction of $\mathcal{T}'_i$ that $\mathcal{T}'_i = (\theta_\Omega \circ pr_\Omega(\mathcal{W}_i)) . pr_\Omega(\mathcal{U}_i)$ (with equality in $W$).
We calculate the following (with all equalities in $W$):
\begin{eqnarray*}
 (\theta \phi)_\Omega (v_i) & = & \mathcal{T}_i v_i \mathcal{T}_i^{-1}\\
 & = & \mathcal{T}'_i v_i (\mathcal{T}'_i)^{-1} \\
 & = & (\theta_\Omega \circ pr_\Omega(\mathcal{W}_i)) . pr_\Omega(\mathcal{U}_i) . v_i . pr_\Omega(\mathcal{U}_i)^{-1} . (\theta_\Omega \circ pr_\Omega(\mathcal{W}_i))^{-1}\\
 & = & \theta_\Omega (pr_\Omega(\mathcal{W}_i)) . \theta_\Omega(v_i) . \theta_\Omega(pr_\Omega(\mathcal{W}_i)^{-1})\\
 & = & \theta_\Omega (pr_\Omega(\mathcal{W}_i). v_i . pr_\Omega(\mathcal{W}_i)^{-1})\\
 & = & \theta_\Omega \circ \phi_\Omega (v_i).
\end{eqnarray*}
\end{proof}

We now prove the main result of the subsection.

\begin{proof}[Proof of Theorem \ref{RetractionRewritingLemma}]
By Lemma \ref{MUpsilonIsAHomomorphism}, $(\phi^{-1})_\Omega \circ \phi_\Omega = (\phi^{-1} \circ \phi)_\Omega = id_\Omega = id$ and
$\phi_\Omega$ is an automorphism of $W$.  So the map $\phi \mapsto \phi_\Omega$ is a map $\AutZ \to \AutZ$.
It follows from Lemma \ref{MUpsilonIsAHomomorphism} that $\phi \mapsto \phi_\Omega$ is a homomorphism $\AutZ \to \AutZ$.
It is clear from the definitions that $$(\chi_{iK})_\Omega  = \left\{\begin{array}{cl}
                                                             \chi_{iK} & \text{if } \chi_{iK} \in \PC_\Omega \\
                                                                id & \text{if } \chi_{iK} \not \in \PC_\Omega. \\
                                                              \end{array}\right.$$
It follows that $\phi \mapsto \phi_\Upsilon$ is a retraction homomorphism  $\AutZ \to \langle \PC_\Omega \rangle$.
\end{proof}

We conclude this subsection by noting the following immediate application of The Restricted Alphabet Rewriting Lemma.

\begin{cor}\label{PCMinimal}
The set $\PC$ is a minimal generating set for $\AutZ$.
\end{cor}

\begin{proof}
Let $\chi_{iK} \in \PC$ and let $\U$ be a
word in the alphabet $\PC^{\pm 1}$ such that
$\U = \chi_{iK}$ (with equality in $\AutZ$). It follows from the Restricted Alphabet Rewriting
Lemma that, simply by omitting some letters, $\U$ may be rewritten
as a word $\U'$ in the alphabet
$$\{\chi_{iQ} \; | \; Q \hbox{ a connected component of } \Gamma \setminus S_i\}^{\pm 1}.$$
But the letters in $\U'$ commute pairwise, so we must have
that $\chi_{iK}$ appears with exponent sum 1 in $\U'$ and hence also
with exponent sum 1 in $\U$.  Thus no word in the alphabet
$$\Bigl(\PC \setminus \{\chi_{iK}\}\Bigr)^{\pm 1}$$
can spell $\chi_{iK}$.
\end{proof}

\subsection{A splitting of $\AutZ$}\label{SplittingAutZSection}

We now define a subset $\PCMinus \subset \PC$.  We will show that $\OutZ$,
the subgroup generated by $\PCMinus$, does not include any non-trivial inner automorphisms
and that it is isomorphic to $\OutZQ$.  Informally, one might understand the construction of $\PCMinus$ from $\PC$
as removing `just enough' automorphisms to prevent the elements of $(\PCMinus)^{\pm 1}$ from spelling a non-trivial inner-automorphism.

Let $\basicinners$ denote the following set of inner automorphisms $$\basicinners = \{w \mapsto v_i w v_i^{-1} \; | \; 1 \leq i \leq N \hbox{ and } \Gamma \setminus S_i \neq \emptyset\}.$$
It is clear that $\basicinners$ generates $\InnW$. The commuting product $\prod_K \chi_{iK}$, taken over all connected components $K$ of $\Gamma \setminus S_i$, is the inner automorphism
$(w \mapsto v_i w v_i^{-1}) \in \basicinners$.
If, starting with $\PC$, we remove one $\chi_{iK}$ for each $i$ such that $\Gamma \setminus S_i \neq \emptyset$, then the union of the resulting set
and $\basicinners$ is a generating set for $\AutZ$.  We now do so systematically.

\begin{defn}[$\PCMinus$ and $\OutZ$]\label{DefnOfQ}
For each $1 \leq i \leq N$ such that $\Gamma \setminus S_i \neq \emptyset$, let
$j_i$ be minimal such that $v_{j_i} \in \Gamma \setminus S_i$.
Define
$$\PCMinus := \{\chi_{i K} \in \PC \; | \; v_{j_i} \not \in K \}.$$
Write $\OutZ$ for the subgroup of $\AutZ$ generated by
$\PCMinus$.  (In $\S$\ref{ExampleSection} we write down $\PCMinus$ for an example.)
\end{defn}

\begin{rem}\label{RemarkOnQ}
Observe the following properties of $\PCMinus$:
\begin{enumerate}
\item As in Corollary \ref{PCMinimal}, the Restricted Alphabet Rewriting
Lemma may be used to show that the set $\basicinners \cup \PCMinus$
is a minimal generating set for $\AutZ$.
\item For each $1 \leq i \leq N$, either $\Gamma \setminus S_i =
\emptyset$ or the set $\PC \setminus \PCMinus$
contains exactly one element of the form $\chi_{i K}$.
\item \label{PropertyOfQ} If $\chi_{i K} \in \PCMinus$, then $v_1 \not \in K$.
If $\chi_{i K} \in \PCMinus$ and $d(v_1, v_i) \leq 1$, then $v_2 \not
\in K$.  In general, if $\chi_{i K} \in \PCMinus$ and $d(v_j, v_i) \leq 1$ for each $1 \leq j \leq k,$ then $v_{k+1} \not \in K$.
\item The set $\PCMinus$ depends on the ordering of $\vertices$
defined by the indexing.  For the work in this section, the ordering
is unimportant.
\end{enumerate}
\end{rem}

\begin{lem}\label{CInnerAndQ}
$\OutZ \cap \Inner = \{id\}.$
\end{lem}

\begin{proof}
For each $1 \leq i \leq N$, write
$\PCT_i := \{\chi_{j K} \in \PCMinus \; | \; v_j \in S_i\}.$  Write $\PCT := {\cap}_{i=1}^N \PCT_i$.
Suppose that $\phi \in \Inner \cap \OutZ$, say $\phi(v_j) = w v_j w^{-1}$ for each $1 \leq j \leq N$.  We shall use induction to show that
$\phi \in \langle \PCT \rangle$.

Since $\phi \in \OutZ$, $\phi$ may be written as a word $\Phi_0$ in the alphabet $(\PCMinus)^{\pm 1}$.
By Remark \ref{RemarkOnQ}(\ref{PropertyOfQ}), each element of
$\PCMinus$ acts trivially on $v_1$. It follows from The Deletion
Condition that $w$ is in the centralizer of $v_1$. By Lemma \ref{CentralizerOfAVertex}, $w \in W(S_1)$. By
the Restricted Alphabet Rewriting Lemma, $\phi$ may be written as a
product $\Phi_1$ in the alphabet $\PCT_1^{\pm 1}$ (starting with
$\Phi_0$, delete those letters not in $\PCT_{1}^{\pm 1}$). Now let $i$
be an integer such that $1 \leq i < N$ and suppose that $\phi$ may
be written as a product $\Phi_i$ in the alphabet $\bigl(\PCT_1\cap
\dots \cap \PCT_i \bigr)^{\pm 1}.$ By Remark
\ref{RemarkOnQ}(\ref{PropertyOfQ}), each element of $\PCT_1\cap
\dots \cap \PCT_i$ acts trivially on $v_{i+1}$. It follows that $w$
is in the centralizer of $v_{i+1}$.  Hence $w \in W(S_{i+1})$.  By the
Restricted Alphabet Rewriting Lemma, $\phi$ may be written as a
product $\Phi_{i+1}$ in the alphabet $\bigl(\PCT_1\cap \dots \cap
\PCT_i \cap \PCT_{i+1} \bigr)^{\pm 1}$ (starting with $\Phi_i$,
delete those letters not in $\PCT_{i+1}^{\pm 1}$).  By
induction we have that $\phi$ may be written as a product $\Phi_{N}$
in the alphabet $\PCT^{\pm 1}$.

Now  $\chi_{j K} \in \PCT$ if and only if $v_j$ is adjacent to each
vertex in $\Gamma$. But for such $v_j$, $\Gamma \setminus S_j =
\emptyset$ and $\PC$ (and hence $\PCMinus$)
contains no partial conjugations with operating letter
$v_j$.  Thus $\PCT = \emptyset$, $\Phi_{N}$ is the empty word and
$\phi$ is the trivial automorphism.
\end{proof}

\begin{proof}[Proof of Theorem \ref{CSemiDirect}]
This follows immediately from Lemma \ref{CInnerAndQ}, the fact that
$\Inner$ is a normal subgroup of $\AutZ$ and the fact that $\I \sqcup \PCMinus
$ generates $\AutZ$.
\end{proof}

\section{Some subgroups of $\AutZ$}\label{ConnectedSubSection}

In this section we prove Theorem \ref{QiEmbedsInAutZ}, which anticipates our study of $\OutZ$.

We first define the link points of a partial conjugation and some associated subsets of $\PC$.
\begin{defn}\label{LinkPointDefn}
For a partial conjugation $\chi_{j Q} \in \PC$ and a vertex $v_i$, we say that $v_i$ is a \emph{link point of} $\chi_{j Q}$ if
$v_j \in L_i$ and $Q \cap L_i \neq \emptyset$.  We write $$\PartitionSet_i := \{\chi_j Q \in \PC \; | \; v_i \hbox{ is a link point of } \chi_{jQ}\}.$$
\end{defn}
\begin{eg}
For example, $v_4$ is the unique link point of $\chi_{2 \{v_8, v_{15}, v_{16}\}}$ in the example examined in $\S$\ref{ExampleSection}.
\end{eg}
For an element $g \in W$, we write $\supp{g}$ for
the minimal subset of $V$ such that $g$ may be written as a word in the alphabet $(\supp{g})^{\pm 1}$ (cf. Lemma \ref{EquivalenceOfReducedWords}).

\begin{proof}[Proof of Theorem \ref{QiEmbedsInAutZ}]
Let $1 \leq i \leq N$, let $\Gamma'$ be the connected component of $\Gamma$ which contains $v_i$ and
let $\phi \in \langle \PartitionSet_i \rangle$ be such that $\phi$ acts as
the identity on $W(L_i)$.   It is immediate from the definitions that $\phi(v_i) = v_i$ and $\phi(v_k) = v_k$ for each
$v_k \in S_i \cup \Gamma \setminus \Gamma'$.  Fix
$j$ such that $v_j \in \Gamma' \setminus S_i$.  It suffices to show
that $\phi(v_j) = v_j$.

Define
\begin{eqnarray*}
H_i^j & := & \{v_\ell \in L_i \; | \; \hbox{each path from } L_i
\hbox{ to } v_j \hbox{ passes through } S_\ell\} \\
\mathcal{H}_i^j & := & \{\chi_{\ell K} \in  \PartitionSet_i \; | \; v_\ell \in
H_i^j\}.
\end{eqnarray*}
Observe that if $\chi_{\ell K} \in
\mathcal{H}_i^j$, then $v_j$ and $L_i \setminus S_\ell$ are in distinct connected components of $\Gamma \setminus S_\ell$; hence
$v_j \not \in K$ and $\chi_{\ell K}(v_j) = v_j$.
Thus $\psi(v_j) = v_j$ for each $\psi \in \langle \mathcal{H}_i^j \rangle$.



In this paragraph we prove that there exists
an element $w \in \langle H_i^j \rangle$ such that $\phi(v_j) = w v_j w^{-1}$.  Let $1 \leq m \leq N$ be such that
$v_m \neq v_j$ and $v_m \in \supp{\phi(v_j)}$. Since $\phi \in \langle \PartitionSet_i \rangle$, $v_m \in L_i$.
We must show that each path from $L_i$ to $v_j$ passes through $S_m$ (and hence $v_m \in H_i^j$). Let $v_{j_1}, v_{j_2}, \dots,
v_{j_\ell} \in \vertices$ be the successive vertices of a path from $v_{j_1} \in L_i$ to $v_{j_\ell} = v_j$.
Let $w_{j_1}, w_{j_2}, \dots, w_{j_\ell} \in W$ be minimal length words such
that $\phi(v_{j_k}) = w_{j_k} v_{j_k} w_{j_k}^{-1}$.  By hypothesis, $\phi$
acts as the identity on $W(L_i)$ and $w_{j_1} = 1$. Also by hypothesis,
$v_m \in \supp{w_{j_\ell}}$.  Let $k$ be minimal such
that $v_m \not \in \supp{w_{k-1}}$ and $v_m \in \supp{w_k}$.
Since $v_{j_{k-1}}$ and $v_{j_k}$ commute and $\phi$ is an
automorphism, we have that $w_{j_{k-1}} v_{j_{k-1}} w_{j_{k-1}}^{-1}$ and
$w_{j_k} v_{j_k} w_{j_k}^{-1}$ commute.  It follows that $w_{j_k}^{-1} w_{j_{k-1}} v_{j_{k-1}} w_{j_{k-1}}^{-1} w_{j_k}$ is in the centralizer of $v_{j_k}$,
which equals $W(S_{j_k})$ by Lemma \ref{CentralizerOfAVertex}.
The following facts are consequences of the Deletion Condition:
\begin{enumerate}
\item $v_m \in \supp{w_{j_k}^{-1} w_{j_{k-1}}}$;
\item if $v_m \not \in \supp{w_{j_k}^{-1} w_{j_{k-1}} v_{j_{k-1}} w_{j_{k-1}}^{-1} w_{j_k}}$, then $v_m \in S_{j_{k-1}}$;
\item if $v_m \in \supp{w_{j_k}^{-1} w_{j_{k-1}} v_{j_{k-1}} w_{j_{k-1}}^{-1} w_{j_k}}$, then $v_m \in S_{j_k}$.
\end{enumerate}
Hence $d(v_{j_{k-1}}, v_m) \leq 1$ or $d(v_{j_k}, v_m) \leq 1$ and the path $v_{j_1}, v_{j_2}, \dots, v_{j_\ell}$ passes
through $S_m$, as required

By the paragraph above and Lemma \ref{RetractionRewritingLemma}, there exists $\psi \in
\langle \mathcal{H}_i^j \rangle$ such that $\psi(v_j) =
\phi(v_j)$ (note that this equality need not hold for all $j$). Hence $\phi(v_j)= v_j$ as required.
\end{proof}

\section{The group $\OutZ$}\label{CardinalityOfOutZ}

In this section we prove Theorem \ref{OutZAbelianIffSIL} and Corollary \ref{OutWInfiniteIffSIL}.
We first investigate the ways in which the connected components of $\Gamma \setminus S_i$ and $\Gamma \setminus S_j$ may interact.

\begin{lem}\label{Distance2TNotInK}
Let $\chi_{i K}, \chi_{j Q} \in \PC$.  If
$d(v_i, v_j) \geq 2$ and $v_j \not \in K$, then $K \cap Q =
\emptyset$ or $K \subset Q$.
\end{lem}

\begin{proof}
Assume that $d(v_i, v_j) \geq 2$, $v_j \not \in K$ and $K \cap Q \neq \emptyset$.  Suppose that $K \not \subset Q$.  Let
$v_m \in K \cap Q$, let $v_k \in K \setminus Q$ and let $\alpha$ be a
path in $K$ from $v_m$ to $v_k$. Since $v_m \in Q$ but $v_k \not \in Q$, there
exists a vertex $v_a$ on $\alpha$ such that $d(v_j, v_a) = 1$.
Since $v_j, v_a \in \Gamma \setminus S_i$ and $d(v_j, v_a) = 1$, the vertices $v_a$ and $v_j$ are contained in the same connected
component of $\Gamma \setminus S_i$.  Hence
$v_j \in K$, contradicting the hypothesis.
\end{proof}

\begin{lem}\label{GeometryOfPCsLemma}
Let $\chi_{i K}, \chi_{j Q} \in \PC$. If $\Gamma$ is connected, then exactly
one of the following thirteen cases holds:
\begin{enumerate}
\item [(1)] $d(v_i, v_j) \leq 1$;
\item [(2)] $d(v_i, v_j) = 2$, $v_i \in Q$, $v_j \in K$, $K \cap Q = \emptyset$;
\item [(3)] $d(v_i, v_j) = 2$, $v_i \in Q$, $v_j \in K$, $K \cap Q \neq \emptyset$;
\item [(4)] $d(v_i, v_j) = 2$, $v_i \in Q$, $v_j \not \in K$, $K \cap Q = \emptyset$;
\item [(5)] $d(v_i, v_j) = 2$, $v_i \in Q$, $v_j \not \in K$, $K \subset Q$;
\item [(6)] $d(v_i, v_j) = 2$, $v_i \not \in Q$, $v_j \in K$, $K \cap Q = \emptyset$;
\item [(7)] $d(v_i, v_j) = 2$, $v_i \not \in Q$, $v_j \in K$, $K \supset Q$;
\item [(8)] $d(v_i, v_j) = 2$, $v_i \not \in Q$, $v_j \not \in K$, $K \cap Q = \emptyset$;
\item [(9)] $d(v_i, v_j) = 2$, $v_i \not \in Q$, $v_j \not \in K$, $K = Q$.
\item [(10)] $d(v_i, v_j) \geq 3$, $v_i \not \in Q$, $v_j \not \in K$, $K \cap Q = \emptyset$;
\item [(11)] $d(v_i, v_j) \geq 3$, $v_i \in Q$, $v_j \not \in K$, $K \subset Q$;
\item [(12)] $d(v_i, v_j) \geq 3$, $v_i \not \in Q$, $v_j \in K$, $K \supset Q$;
\item [(13)] $d(v_i, v_j) \geq 3$, $v_i \in Q$, $v_j \in K$, $K \cup Q = \Gamma$.
\end{enumerate}
The relation $\chi_{iK} \chi_{j Q} = \chi_{j Q} \chi_{iK}$ holds in
cases (1), (5), (7), (8), (10), (11) and (12). The relation
$\chi_{iK} \chi_{j Q} = \chi_{j Q} \chi_{iK}$ fails in cases (2), (3),
(4), (6), (9) and (13).
\end{lem}

\begin{proof}
It follows immediately from Lemma \ref{Distance2TNotInK} that the
cases (1)-(9) are an exhaustive list of the possibilities when
$d(v_i, v_j) \leq 2$. Thus we may assume that $d(v_i, v_j) \geq 3$.
\begin{description}
\item [Case $v_i \not \in Q$, $v_j \not \in K$] By Lemma \ref{Distance2TNotInK}, either $K \cap Q = \emptyset$ or $K = Q$.
Suppose that $K = Q$.  Let $v_k \in K$ (hence $v_k \in Q$) be such that
$d(v_i, v_k) = 2$ and let $v_{k'} \in \vertices$ be such that $d(v_i, v_{k'})
= d(v_{k'}, v_k) = 1$. By the triangle inequality, $d(v_j, v_{k'}) \geq 2$.
Since $d(v_k, v_{k'}) = 1$, $v_k$ and $v_{k'}$ are in the same connected
component of $\Gamma \setminus S_j$.  Thus $v_{k'} \in Q = K$, a
contradiction to the fact that $d(v_i, v_{k'}) = 1$.  Hence $K \cap Q =
\emptyset$.
\item [Case $v_i \in Q$ and $v_j \not \in K$]
Let $v_k \in K$ be such that $d(v_i, v_k) = 2$ and let $v_{k'}$ be such that $d(v_i, v_{k'}) = d(v_{k'}, v_k) = 1$.
By the triangle inequality, $d(v_j, v_{k'}) \geq 2$ and $d(v_j, v_k) \geq 1$.  Since $v_j \not \in K$, $d(v_j, v_k) > 1$.
Since $v_i, v_{k'}, v_k \in \Gamma \setminus S_i$ and $d(v_i, v_{k'}) = d(v_{k'}, v_k)=1$, the vertices $v_i, v_{k'}$ and $v_k$ are contained in the same connected component of
$\Gamma \setminus S_j$.  Thus $v_k \in Q$ and $K \cap Q \neq \emptyset$.  By Lemma \ref{Distance2TNotInK}, $K \subset Q$.
\item [Case $v_i \not \in Q$ and $v_j \in K$]
The proof is similar to the case $v_i \in Q$ and $v_j \not \in K$
above.
\item [Case $v_i \in Q$ and $v_j \in K$]
Let $v_c$ be a vertex in $\Gamma \setminus K$. Let $\alpha$ be a minimal length
path from $v_c$ to $v_j$. Since $v_j \in K$ and $v_c \not \in K$,
there exists a vertex $v_a$ on $\alpha$ such that $d(v_a, v_i) \leq 1$.  By the triangle inequality, $d(v_a, v_j) \geq
2$.  It follows that $d(v_c, v_j) \geq 2$ also.  Since $v_i, v_a, v_c \in \Gamma \setminus S_j$ and $d(v_i, v_a) \leq 1$ and the subpath of $\alpha$
from $v_a$ to $v_c$ lies in $\Gamma \setminus S_j$, the vertices $v_i$, $v_a$ and $v_c$ are contained in a single connected
component of $\Gamma \setminus S_j$. Hence $v_c \in Q$ and $Q \cup K =
\Gamma$.
\end{description}
We leave the reader to verify the statements about commuting
products.
\end{proof}

\begin{rem}\label{Case13ImpossibleRemark}
Assume that $\Gamma$ is connected and let $\chi_{i K}, \chi_{j Q} \in \PCMinus$.  It follows from the definition of $\PCMinus$ that $v_1 \not \in K \cup Q$ and Case (13) of Lemma \ref{GeometryOfPCsLemma}
is impossible.  Hence if $d(v_i, v_j) \neq 2$, then the relation $\chi_{iK} \chi_{j Q} = \chi_{j Q} \chi_{iK}$ holds.
\end{rem}

\begin{lem}\label{CCofIndividualAndCombined}
Let $1 \leq i < j \leq N$ be such that $d(v_i, v_j) \geq 2$ and let $R$ be a subgraph of $\Gamma$.  Then $R$ is a connected component
of $\Gamma \setminus S_i$ and $\Gamma \setminus S_j$ if and only if $R$ is a connected component of $\Gamma \setminus (L_i \cap L_j)$
and $v_i, v_j \not \in R$.
\end{lem}

\begin{proof}

If $d(v_i, v_j) \geq 3$, then $L_i \cap L_j = \emptyset$ and each connected component of $\Gamma \setminus (L_i \cap L_j)$ is a connected component
of $\Gamma$.  If $R$ is a connected component of $\Gamma \setminus S_i$ and $\Gamma \setminus S_j$, then $v_i, v_j \not \in R$ and it follows from
Lemma \ref{GeometryOfPCsLemma} that $R$ is a connected component of $\Gamma$. The result follows.

Now assume that $d(v_i, v_j) = 2$.  Let $\Gamma'$ denote the connected component of $\Gamma$ which contains $v_i$ and $v_j$.
If $R$ is not a subgraph of $\Gamma'$, then $R$ is a connected component
of $\Gamma \setminus S_i$ and $\Gamma \setminus S_j$ if and only if $R$ is a connected component of $\Gamma$
and $v_i, v_j \not \in R$.  The result follows. Assume that $R$ is a connected component of $\Gamma' \setminus S_i$ and $\Gamma' \setminus S_j$.  Clearly, $v_i, v_j \not \in R$.
Since $R$ is a connected subgraph of $\Gamma' \setminus S_i$ and $L_i \cap L_j \subset S_i$, $R$ is a connected subgraph of
$\Gamma' \setminus (L_i \cap L_j)$.  Suppose that $R$ is not a connected component of $\Gamma' \setminus (L_i \cap L_j)$.  Then there exist $v_x \in R$,
$v_y \in \Gamma' \setminus (R \cup (L_i \cap L_j))$ such that $d(v_x, v_y) = 1$.  Since $v_x \in R$ and $v_y \not \in R$ and $R$ is a connected component of
$\Gamma' \setminus S_i$, $v_y \in S_i$. Similarly, $v_y \in S_j$.  Thus $v_y \in S_i \cap S_j = L_i \cap L_j$---a contradiction.
Hence $R$ is a connected component of $\Gamma' \setminus (L_i \cap L_j)$.


Now assume that $d(v_i, v_j) = 2$ and $R$ is a connected component of $\Gamma' \setminus (L_i \cap L_j)$ and $v_i, v_j \not \in R$.  Since $v_i \not \in R$,
$S_i \cap R = \emptyset$ and $R$ is a connected subgraph of $\Gamma' \setminus S_i$.  Suppose that $R$ is not a connected component of $\Gamma' \setminus S_i$.
Then there exist $v_x \in R$, $v_y \in \Gamma' \setminus (R \cup S_i)$ such that $d(v_x, v_y) = 1$.  Since $v_x \in R$ and $v_y \not \in R$ and
$R$ is a connected component of $\Gamma' \setminus (L_i \cap L_j)$, $v_y \in L_i \cap L_j \subset S_i$---a contradiction.
Hence $R$ is a connected component of $\Gamma' \setminus S_i$. Similarly,
$R$ is a connected component of $\Gamma' \setminus S_j$.
\end{proof}

Recall the definition of an SIL (Definition \ref{SILDefn}).

\begin{lem}\label{KQLinkEverythingLemma}
Assume that $\Gamma$ does not contain a SIL.  Let $1 \leq i < j \leq N$ be such that $d(v_i, v_j) =2$, let $K_j$ be
the connected component of $\Gamma \setminus S_i$ which contains $v_j$ and let $Q_i$ be the connected component of
$\Gamma \setminus S_j$ which contains $v_i$.  Then
$\Gamma = K_j \cup Q_i \cup (L_i \cap L_j)$.
\end{lem}

\begin{proof}
Since $\Gamma$ does not contain a SIL, $\Gamma \setminus (L_i \cap L_j)$ has at most two connected components.
If $\Gamma \setminus (L_i \cap L_j)$ has two connected components, they are $K_j$ and $Q_i$ and the result is clear.  Assume that
$\Gamma \setminus (L_i \cap L_j)$ is connected.  Let $v_x \in \Gamma \setminus (L_i \cap L_j)$ and let $\alpha$ be a minimal length path in
$\Gamma \setminus (L_i \cap L_j)$ from $v_x$ to $v_i$.  If $\alpha$ passes through $S_j$, then $v_x \in K_j$.
If $\alpha$ does not pass through $S_j$, then $v_x \in Q_i$.  Hence the result.
\end{proof}

Combined with Remark \ref{WLOGAssumeConnected}, the following lemma allows us to assume that $\Gamma$ is connected
when proving Theorem \ref{OutZAbelianIffSIL} and Corollary \ref{OutWInfiniteIffSIL}.  The lemma is immediate.

\begin{lem}\label{SILPlusIFFSIL}
Let $\Gamma^+$ be as in Remark \ref{WLOGAssumeConnected}.  Then $\Gamma$ has a SIL if and only if $\Gamma^+$ has a SIL.
\end{lem}

We now prove the two main results of the section.

\begin{proof}[Proof of Theorem \ref{OutZAbelianIffSIL}]
Assume that $W$ is a graph product of directly-inde-composable
cyclic groups.  By Remark \ref{WLOGAssumeConnected} and Lemma \ref{SILPlusIFFSIL}, we may assume without loss that $\Gamma$ is connected.

Suppose $\Gamma$ contains a SIL with $i$, $j$ and $R$ as in
Definition \ref{SILDefn}.  Let $K$ denote the connected component of $\Gamma \setminus S_i$ which contains $v_j$ and let $Q$ denote the
the connected component of $\Gamma \setminus S_j$ which contains $v_i$.  By Lemma \ref{CCofIndividualAndCombined},
$\chi_{iR}, \chi_{jR} \in \PC$.  If $R$ does not contain the least element of $\Gamma \setminus L_i \cap L_j$, then
$\chi_{iR}, \chi_{jR} \in \PCMinus$.  If $R$ does contain the least element of $\Gamma \setminus L_i \cap L_j$, then
$\chi_{iK}, \chi_{jQ} \in \PCMinus$.  Calculation confirms that
$\chi_{iR} \chi_{jR} \neq \chi_{jR} \chi_{iR}$ and $\chi_{iK} \chi_{jQ} \neq \chi_{jQ} \chi_{iK}$.  Hence $\OutZ$ is not abelian and Property (1) implies Property (2).

Now suppose that $\Gamma$ does not contain a SIL and let $\chi_{i K},
\chi_{j Q} \in \PCMinus$.  By Remark \ref{Case13ImpossibleRemark}, the relation $\chi_{iK} \chi_{jQ} = \chi_{jQ} \chi_{iK}$ holds whenever
$d(v_i, v_j) \neq 2$.  Assume that $d(v_i, v_j) = 2$.
By Lemma \ref{KQLinkEverythingLemma}, $\Gamma = K_j \cup Q_i \cup (L_i \cap L_j)$ for $K_j$ and $Q_i$ as in the statement of the Lemma.  Without loss, assume that the least element
of $\Gamma \setminus (L_i \cap L_j)$ is contained in $K_j$.  By the definition of $\PCMinus$, $K \neq K_j$.  Thus $v_j \not \in K$ and $K \subset Q_i$.  If $Q=Q_i$, then
$K \subset Q$ and case (5) of Lemma \ref{GeometryOfPCsLemma} holds.  If $Q \not = Q_i$, then $v_i \not \in Q$ and $K \cap Q =\emptyset$ and case (8)
of Lemma \ref{GeometryOfPCsLemma} holds.  In either case, the relation $\chi_{iK} \chi_{jQ} = \chi_{jQ} \chi_{iK}$ holds.
Thus $\OutZ$ is an abelian group and Property (2) implies Property (1).
\end{proof}

\begin{proof}[Proof of Corollary \ref{OutWInfiniteIffSIL}]
Assume that $W$ is a graph product of primary
cyclic groups.  By Remark \ref{WLOGAssumeConnected} and Lemma \ref{SILPlusIFFSIL}, we may assume without loss that $\Gamma$ is connected.

Since each partial conjugation has finite order, it is clear that Property (1) implies Property (3).
Suppose $\Gamma$ contains a SIL with $i$, $j$ and $R$ as in
Definition \ref{SILDefn}.  Let $v_r$ be a vertex in $R$.  Calculation confirms that
$(\chi_{iR} \chi_{jR})^n(v_r) = (v_j v_i)^n v_r (v_j v_i)^{-n}$ and $(\chi_{iR} \chi_{jR})^n(v_i) = v_i$ for each positive integer $n$.
It follows that no power of $\chi_{iR} \chi_{jR}$ is an inner automorphism.  Hence $\OutW$, and $\OutZ$, have infinite order and Property (3) implies Property (2).
\end{proof}

\section{The group $\OutZ$ in the case that $\Gamma$ is a tree}\label{TreeSection}

Theorems \ref{TitsSplittingRecovered} and \ref{CSemiDirect} reduce the study of $\CAutW$ to the study of $\OutZ$.  In this section we describe the
structure of $\OutZ$ in the special case that $\Gamma$ is a tree.  The reader may wish to switch back and forth between
this section and $\S$\ref{ExampleSection}, in which we follow a specific example.

Throughout, we assume that $\Gamma$ is a tree with at
least three vertices.  In particular, each $W(L_i)$ is a free product of cyclic groups.  By reindexing, if necessary, we may further assume that
indices have been assigned to elements of $\vertices$ so that
$v_1$ is a leaf (that is, adjacent to exactly one vertex) and
if $d(v_1, v_i) < d(v_1, v_j)$, then $i < j$.

Because $\Gamma$ is a tree, each partial conjugation has a unique link point and we may define a natural
partition of $\PCMinus$ as follows:
for each $i = 1, 2, \dots, N$, define
$$\PartitionSet^0_i :=  \{\chi_{j Q} \in \PCMinus \; | \; v_i \hbox{ is the link point of } \chi_{j Q}\}.$$
Observe the following properties:
\begin{enumerate}
\item $\PartitionSet^0_i = \PartitionSet_i \cap \PCMinus$;
\item $\PCMinus  = \PartitionSet^0_N \sqcup \PartitionSet^0_{N-1} \sqcup \dots \sqcup \PartitionSet^0_1$;
\item if $v_i$ is a leaf, then $\PartitionSet^0_i = \emptyset$.
\end{enumerate}

\begin{lem}\label{LisCommuteInTree}
If $\chi_{iK}, \chi_{j Q} \in \PCMinus$ are in distinct elements of the partition $\PartitionSet^0_N \sqcup \PartitionSet^0_{N-1} \sqcup \dots \sqcup \PartitionSet^0_1$,
then $\chi_{iK}$ and $\chi_{j Q}$ commute.
\end{lem}

\begin{proof}
Let $\chi_{iK}, \chi_{j Q} \in \PCMinus$ be elements which do not commute.
By Remark \ref{Case13ImpossibleRemark}, one of cases (2), (3), (4), (6) or (9) in Lemma \ref{GeometryOfPCsLemma}
must hold.  We leave the reader to verify that the definition of $\PCMinus$ and
the simple geometry of a tree imply that case (3) is impossible, and cases (2), (4), (6) and (9) may only hold if $\chi_{iK}, \chi_{j Q}$
have a common link point.  Thus the result.
\end{proof}

\begin{cor}\label{DirectProductStructureForTree}
$\OutZ = \langle \PartitionSet^0_N \rangle \times \langle \PartitionSet^0_{N-1} \rangle \times \dots \times \langle \PartitionSet^0_1 \rangle.$
\end{cor}

The following proposition completes the proof of Theorem \ref{NiceStructureOfOutZForTree}.  In the statements below, we write
$\Integer_{\ordermap(j)}$ for the cyclic group of order $m(j)$.

\begin{prop}\label{QiInCaseGammaATree}
Suppose that $\Gamma$ is a tree with at least three vertices.  Let $1 \leq i \leq N$ and
let $L_i = \{v_{k_1}, v_{k_2}, \dots, v_{k_M}\}$ with $k_1 < k_2 < \dots < k_M$.
If $M = 1$ (that is, $v_i$ is a leaf) or $M > 1$ and $v_{k_2}$ is the minimal element of $\Gamma \setminus S_{k_1}$, then
$$\langle \PartitionSet^0_i \rangle \cong \OutZ(L_i);$$
otherwise,
$$\langle \PartitionSet^0_i \rangle \cong \Integer_{\ordermap(k_1)} \times \OutZ(L_i).$$
\end{prop}

\begin{proof}
If $M = 1$, then $\PartitionSet^0_i = \emptyset$, $\OutZ(L_i)$ is trivial and the result holds.
So we may assume that $M > 1$.
Let $\rho_i: \langle \PartitionSet^0_i \rangle \to \AutW(L_i)$ denote the homomorphism determined by restriction, that is,
$\chi_{iK} \mapsto \chi_{iK}|_{L_i}.$  Theorem \ref{QiEmbedsInAutZ} gives that $\rho_i$ is injective.  We must show that the image
$\rho_i(\langle \PartitionSet^0_i \rangle)$ is as described in the
conclusion of the Proposition.

Assume that the minimal element of $\Gamma \setminus S_{k_1}$ is $v_{k_2}$.  Using the notation
$\chi_{k_j \{k_\ell\}}:= \chi_{k_j \{v_{k_\ell}\}}$, the image $\rho_i(\PartitionSet^0_i)$ is as follows:
$$\rho_i(\PartitionSet^0_i) = \{\chi_{k_j \{k_\ell\}} \; | \; 1 \leq j, \ell \leq M, j \neq \ell\} \setminus \{\chi_{k_1 \{k_2\}}, \chi_{k_2 \{k_1\}}, \dots, \chi_{k_M \{k_1\}} \}.$$
This is a generating set for $\OutZ(L_i)$.

Now assume that $v_{k_2}$ is not the minimal element of $\Gamma \setminus S_{k_1}$ (so the minimal element of $\Gamma \setminus S_{k_1}$ is not contained in $L_i$). The image $\rho_i(\PartitionSet^0_i)$ is as follows:
$$\rho_i(\PartitionSet^0_i) = \{\chi_{k_j \{k_\ell\}} \; | \; 1 \leq j, \ell \leq M, j \neq \ell\} \setminus \{\chi_{k_2 \{k_1\}}, \dots, \chi_{k_M \{k_1\}} \}.$$
If we replace $\chi_{k_1 \{k_2\}}$ by the product $\chi_{k_1 \{k_2\}} \dots \chi_{k_1 \{k_M\}}$, then the resulting set still generates
$\langle \rho_i(\PartitionSet^0_i) \rangle$.  Observe that $\chi_{k_1 \{k_2\}} \dots \chi_{k_1 \{k_M\}}$ commutes with each element in the set
$\rho_i(\PartitionSet^0_i) \setminus \{\chi_{k_1 \{k_2\}}\}$
and $\rho_i(\PartitionSet^0_i) \setminus \{\chi_{k_1 \{k_2\}}\}$ generates $\OutZ(L_i)$.  Thus $\rho_i(\PartitionSet^0_i)$ generates a subgroup of $\AutZ(L_i)$ which is isomorphic to $\Integer_{\ordermap(k_1)} \times \OutZ(L_i)$.
\end{proof}

\begin{rem}
Our hypotheses on the indexing of $\vertices$ ensure that $v_1$ and $v_N$ are leaves.
One may omit the corresponding terms $\OutZ(L_1)$ and $\OutZ(L_N)$ (and
any other terms corresponding to leaves) from the statement of Theorem \ref{NiceStructureOfOutZForTree}.  However, by including these terms we ensure that
the statement stays valid even if the hypotheses
on the indexing is dropped.
\end{rem}

\begin{rem}[Presenting $\CAutW$ in the case that $\Gamma$ is a tree]\label{PresentingCAutWRemark}
Since $\Ab$ is the direct product $\prod_{k \in K} \; \Integer_{\ordermap(k)}$, where $K$ is the set
\begin{multline*}
\{k \; | \; \exists i \;\; 1 \leq i, k \leq N, v_k \hbox{ is the minimal vertex in } L_i \hbox{ and } \\L_i \hbox{ does not contain the minimal vertex of } \Gamma \setminus S_k\},
\end{multline*}
one may write down a finite presentation of $\Ab$.  Since
each $W(L_i)$ is a free product of cyclic groups, one may use work
of Fouxe-Rabinovitch \cite{FRI} (see also \cite[footnote 1,
p.1]{McCulloghAndMiller}) and Gilbert \cite{GilbertAutOfFreeProduct}
to write down a finite presentation for $\OutZ(L_i)$.  Combining
these presentations in the standard way for presenting a direct
product gives a finite presentation for $\OutZ$.
Further, a finite presentation for $\InnW$ is well-known (cf. Lemma \ref{CenterOfW}) and, because each maximal complete subgroup is a direct product of two
cyclic groups, it is an easy exercise to write
down a finite presentation of $\QAutW$.  Combining the presentations of $\InnW$, $\OutZ$ and $\QAutW$ in the standard way for
presenting semi-direct products (including computing the image of
each generator of the normal factor under conjugation by each
generator of the other factor) one is then able to write down a
finite presentation of $\CAutW$ (cf. \cite{LaurenceThesis} \cite{Bernie}).
\end{rem}

\begin{rem}\label{SometimesSemiDirectProduct}
Consider the case that $\Gamma$ is an arbitrary connected graph.  Without loss of generality, assume that indices have been assigned to
elements of $\vertices$ so that if $d(v_1, v_i) < d(v_1, v_j)$, then $i < j$.
Unlike the tree case, a partial conjugation may have more than one link point and the sets $\PartitionSet_i^0$
do not partition $\PCMinus$.  However, taking inspiration from the tree case, we define a
partition $\PCMinus$ inductively as follows: write $\M_1 := \PCMinus$ and for each $1 \leq i \leq N$,
\begin{eqnarray*}
\PartitionSet'_i  & := & \{\chi_{j Q} \in \M_i \; | \; v_i \hbox{ is a link point of } \chi_{j Q}\} \\
\M_{i+1}  & := & \M_{i} \setminus \PartitionSet'_i.
\end{eqnarray*}
In some cases, but not all, this partition corresponds to a semi-direct product decomposition of $\OutZ$.
\end{rem}

\section{Applications}\label{ApplicationsSection}

In this section we describe a number of applications of the results above.
We begin with some applications of Corollary \ref{OutWInfiniteIffSIL}.

For a one ended word hyperbolic group
$G$, $\Out(G)$ is infinite if and only if $G$ splits over a
virtually cyclic subgroup with infinite center, either as an arbitrary HNN
extension or as an amalgam of groups with finite center \cite{Levitt}.  The
following corollary demonstrates that such splittings are not
possible in the case that $W$ is a graph product of primary cyclic groups.
The proof uses the fact that a graph product of primary cyclic groups is
word hyperbolic if and only if every circuit in $\Gamma$ of length four contains a chord \cite{MeierHyperbolicGraphProducts} and the fact that
each separating subgraph of $\Gamma$ corresponds to a splitting of $W$ as a free product with amalgamation (with the separating
subgraph generating the amalgamated subgroup).

\begin{cor}\label{OutWInfiniteMeansNotHyperbolic}
If $W$ is a graph product of primary cyclic groups and $W$ is a one ended word hyperbolic group, then $\OutW$ is finite.
\end{cor}

\begin{proof}
Let $W$ be a graph product of primary cyclic groups which is one ended and word hyperbolic.  Suppose that $\Gamma$
contains a SIL.  By Theorem \ref{OutWInfiniteIffSIL}, there exist
$i, j, R$ as in Definition \ref{SILDefn}.  If $L_i \cap L_j$ is a complete graph, then $W(L_i \cap L_j)$ is finite and it follows from the Ends Theorem of Hopf and Stallings
(see, for example, \cite[Theorem I.8.32]{MartinsBook}) that
$W$ has infinitely-many ends---a contradiction to the hypothesis.  Thus $L_i \cap L_j$ is not a complete subgraph
and there exist non-adjacent vertices $v_x, v_y \in L_i \cap L_j$. Then $v_i v_x v_j v_y$
is a non-chordal square and $W$ is not word hyperbolic---again, a
contradiction to the hypothesis.
\end{proof}

\begin{rem}\label{ExamplesLikeMNSRemark}
In \cite{ExamplesOfHyperbolicGroups}, the authors construct one-ended hyperbolic groups with finite outer automorphism group and a non-trivial
JSJ decomposition in the sense of Bowditch (that is, the group has a non-trivial graph of groups decomposition with
two-ended edge groups and vertex groups which are either two-ended, maximal ``hanging fuchsian'', or non-elementary
quasiconvex subgroups not of the previous two types---for more details, see \cite{BoCutPts}).
Such groups necessarily have only the trivial JSJ decomposition in the sense of Sela since the outer automorphism groups are finite.
Using Corollary \ref{OutWInfiniteIffSIL}, one may construct examples of right-angled Coxeter groups with similar properties to the groups
described in \cite{ExamplesOfHyperbolicGroups}. In particular, if $W$ is a right-angled Coxeter group and the following properties hold:
\begin{enumerate}
\item $\Gamma \setminus \Delta$ is connected for each complete subgraph $\Delta$;
\item every circuit in $\Gamma$ of length four contains a chord;
\item $\Gamma \setminus \Lambda$ is disconnected for some subgraph $\Lambda$ which generates a virtually abelian group;
\item $\Gamma$ has no SIL;
\end{enumerate}
then $W$ is a one-ended hyperbolic group with a non-trivial JSJ decomposition in the sense of Bowditch and $\OutW$ is finite.
For example, the graph $\Gamma$ of Figure \ref{ExampleLikeMNS} has the desired properties
(with $\Lambda = \{v_1, v_4\}$).
\end{rem}

\begin{figure}
\centering
\includegraphics[scale=0.3]{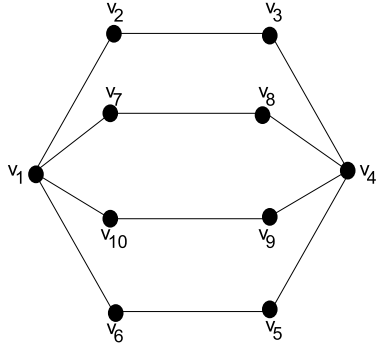}
\caption{The graph $\Gamma$ for Remark \ref{ExamplesLikeMNSRemark}. \label{ExampleLikeMNS}}
\end{figure}

If $W$ is a graph product of primary cyclic groups, then there exists a geometric action of $W$ on a CAT(0) space \cite{MeierHyperbolicGraphProducts}.
We say that $W$ has \emph{isolated flats} if there exists a geometric action of $W$ on a CAT(0) space with
isolated flats (see \cite{HruskaAndKleiner}).
To prove the lemma below we shall need only the following property of such
groups, which follows from the results in \cite{HruskaAndKleiner}:
\begin{enumerate}
\item [($\ast$)] if $W$ has isolated flats and $S_1, S_2 \subseteq W$ are subgroups
isomorphic to $\Integer \times \Integer$ and $S_1 \cap S_2 \neq \{id\}$, then $\langle S_1, S_2 \rangle$ is virtually-abelian.
\end{enumerate}

\begin{lem}\label{IsolatedFlatsLinkLemma}
If $W$ is a graph product of primary cyclic groups and $W$ is one ended with isolated flats and $1 \leq i < j \leq N$ are
such that $d(v_i, v_j) = 2$, then $W(L_i \cap L_j)$ is
virtually abelian.
\end{lem}

\begin{proof}
Let $W$, $i$ and $j$ be as in the hypothesis of the lemma.  Suppose that $W(L_i \cap L_j)$ is
not virtually abelian. Graph products of primary cyclic groups are subgroups of Coxeter groups \cite[Corollary 5.11]{JanusAndSwicat} and hence
linear and satisfy the Tits Alternative.  Further, since $W$ acts geometrically on a CAT(0) space,
each virtually solvable subgroup is virtually abelian \cite[p. 249]{MartinsBook}.  It follows that there exist elements $a, b \in W(L_i \cap L_j)$
such that $\langle a, b \rangle$ is a free group of rank two. The subgroups $S_1=\langle v_iv_j,a\rangle$ and $S_2=\langle v_iv_j,b\rangle$
witness that $W$ does not have property $(\ast)$, since $\langle S_1\cup S_2\rangle$ contains the subgroup $\langle a, b \rangle$.
\end{proof}

\begin{cor}\label{IsolatedFlatsAndOutWInfinite}
If $W$ is a graph product of primary cyclic groups and $W$ is one ended with isolated flats and $\OutW$ is infinite,
then $W$ splits as a free product with amalgamation $W = A \ast_C B$ where
\begin{enumerate}
\item $A$ and $B$ are special subgroups; and
\item $C$ is an infinite virtually abelian special subgroup.
\end{enumerate}
\end{cor}

\begin{proof}[Proof of Corollary \ref{IsolatedFlatsAndOutWInfinite}]
Let $W$ and $\OutW$ be as in the hypothesis of the corollary.
By Theorem \ref{OutWInfiniteIffSIL}, there exist $1 \leq i < j \leq N$
such that $d(v_i, v_j) = 2$ and $L_i \cap L_j$ separates $\Gamma$.  By
Lemma \ref{IsolatedFlatsLinkLemma}, $W(L_i \cap L_j)$ is virtually abelian.  Since $W$ is one-ended, $W(L_i \cap L_j)$ is not finite.
The result follows, with $C = L_i \cap L_j$.
\end{proof}

We say that $W$ has \emph{property (NLC)} if for every CAT(0)
space $X$ on which $W$ acts geometrically, the visual boundary $\partial X$ (see \cite[p. 264]{MartinsBook}) is not locally connected.

\begin{cor}\label{OutWInfiniteMeansNLC}
If $W$ is a right-angled Coxeter group and $\OutW$ is infinite,
then $W$ has property (NLC).
\end{cor}

\begin{proof}
Assume that $\OutW$ is infinite.  By Theorem
\ref{OutWInfiniteIffSIL}, there exist $i, j, R$ as in Definition
\ref{SILDefn}. It follows that $W$ is not finite or two-ended.
If $W$ has infinitely-many ends, then $W$ has property (NLC).
Assume that $W$ is one ended.  Since $(L_i \cap L_j, L_i \cap
L_j, \{v_i, v_j\})$ is a `virtual factor separator' \cite[Definition 3.1]{LocalConnectivityofRACG} and $L_i \cap L_j$ is not a
`suspended separator' \cite[Definition 3.1]{LocalConnectivityofRACG},  we may apply \cite[Theorem 3.2(2)]{LocalConnectivityofRACG} to conclude that
$W$ has property (NLC).
\end{proof}

\begin{rem}\label{NoConverseForNLC}
We now demonstrate that the converse to Corollary \ref{OutWInfiniteMeansNLC} does not hold.
Let $W$ be the right-angled Coxeter group corresponding to the graph $\Gamma$ in
Figure \ref{BadExampleFigure}.  Observe that $\Gamma$ does not contain a
SIL, but $(\{v_2, v_3, v_4\}$, $\{v_2, v_3, v_4\}$, $\{v_1, v_6\})$ is
a virtual factor separator and $\{v_2, v_3, v_4\}$ is not a suspended separator.  Thus
$\OutW$ is finite, by Theorem \ref{OutWInfiniteIffSIL}, and $W$ has
property (NLC), by \cite[Theorem 3.2(2)]{LocalConnectivityofRACG}.
\begin{figure}
\centering
\includegraphics[scale=0.3]{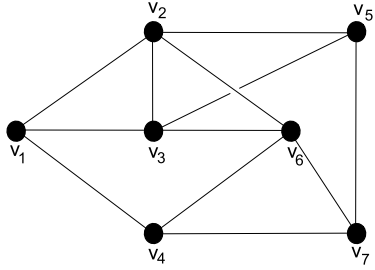}
\caption{The graph $\Gamma$ for Remark \ref{NoConverseForNLC}. \label{BadExampleFigure}}
\end{figure}
\end{rem}

The following application offers a glimpse of some geometry of $\AutW$.

\begin{cor}\label{AutWHyperbolic}
Let $W$ be a graph product of primary cyclic groups. Then $\AutW$ is word hyperbolic if and only if the following conditions are satisfied:
\begin{enumerate}
\item $\Gamma$ has no SIL;
\item every circuit in $\Gamma$ of length four contains a chord.
\end{enumerate}
\end{cor}

\begin{proof}
Assume that $\Gamma$ has an SIL.  Let $i, j, R$ be as in Definition \ref{SILDefn} and let $\iota_{v_i v_j}$ denote the inner automorphism
$w \mapsto v_i v_j w v_j^{-1} v_i^{-1}$.  Since $d(v_i, v_j) \geq 2$, $\iota_{v_i v_j}$ has infinite order.  The product
$\chi_{j R} \chi_{iR}$ also has infinite order and $\langle \iota_{v_i v_j}, (\chi_{j R} \chi_{iR})\rangle \cong \Integer \times \Integer$.  Thus $\AutW$ is not hyperbolic.

Assume that $\Gamma$ has no SIL.  By Corollary \ref{OutWInfiniteIffSIL}, $\OutW$ is finite and $\InnW$ is a
finite-index subgroup of $\AutW$.  But $\InnW$ is also a finite-index subgroup of $W$ (see $\S$\ref{Preliminaries}).  Thus $\AutW$ and $W$ are commensurable,
and hence quasi-isometric (see \cite[Example I.8.8.20(1)]{MartinsBook}).
The result follows immediately
from the characterization of word hyperbolic graph products
of primary cyclic groups described above and the fact that word-hyperbolicity is a quasi-isometry invariant \cite[Theorem III.H.1.9]{MartinsBook}.
\end{proof}

The authors of \cite{MoAndAnton} give sufficient conditions (distinct from those below) for $\AutW$ to split as $\InnW \rtimes \OutW$ in the case that $W$ is a right-angled
Coxeter group.  They then describe an application of their results to group extensions.  We now follow an analogous program for graph products of primary
cyclic groups.  For each $1 \leq i \leq N$, we write
$\Delta_i := \{v_j \in \vertices \; | \; S_i = S_j\}$ (note that $\Delta_i$ is a complete subgraph for each $1 \leq i \leq N$).

\begin{lem}\label{WhenAreSplittingsCompatibleLemma}
If $W$ is a graph product of directly-indecomposable cyclic groups and
$\phi(W(\Delta_i)) = W(\Delta_i)$ for each $1 \leq i \leq N$ and $\phi \in \QAutW$, then the
splittings of Theorems \ref{TitsSplittingRecovered} and \ref{CSemiDirect}
are compatible; that is, one may write $$\CAutW = \InnW \rtimes (\OutZ \rtimes \QAutW) \cong \InnW \rtimes \COutW.$$
\end{lem}

\begin{proof}
Let $\phi \in \QAutW$ and $\theta \in \OutZ$. If $\theta$ is not the identity, then the product $\phi \theta$ is not an inner automorphism since
it acts non-trivially on the set of conjugacy classes of cyclically reduced involutions in $W$.  If $\theta$ is the identity but $\phi$ is not, then the product $\phi \theta = \phi$
is not an inner automorphism by Lemma \ref{CInnerAndQ}.  Thus, to show that $\InnW \cap \COutW = \{id\}$ and hence the result, it suffices to show
that $\COutW = \OutZ.\QAutW$.

Let $\phi \in \QAutW$ and $\chi_{i K} \in \PCMinus$.  It suffices to show that $\phi^{-1} \chi_{i K} \phi \in \langle \PCMinus \rangle$.
Let $1 \leq j \leq N$.  If $v_j \not \in K$, then $\Delta_j \cap K = \emptyset$ and the support of $\phi(v_j)$ is disjoint from $K$. Hence
$$\phi^{-1} \chi_{i K} \phi(v_j) = \phi^{-1} \bigl(\chi_{i K}(\phi(v_j))\bigr) = \phi^{-1} \phi(v_j) = v_j.$$
If $v_j \in K$, then $\Delta_j \subseteq K$ and the support of $\phi(v_j)$ is contained in $K$. Hence
$$\phi^{-1} \chi_{i K} \phi(v_j) = \phi^{-1} \bigl(\chi_{i K} (\phi(v_j))\bigr) = \phi^{-1} (v_i \phi(v_j) v_i^{-1}) = \phi^{-1}(v_i) v_j \bigl(\phi^{-1}(v_i)\bigr)^{-1} .$$
Thus we have
$$\phi^{-1} \chi_{i K} \phi(v_j) = \left\{%
\begin{array}{cl}
  v_j\;\;\;\;\;\;\; & \hbox{if } v_j \not \in K, \\
  \phi^{-1}(v_i) v_j \bigl(\phi^{-1}(v_i)\bigr)^{-1} & \hbox{if } v_j \in K.
\end{array}%
\right.$$
By hypothesis, $\phi^{-1}(v_i) \in W(\Delta_i)$.
For each $v_\ell \in \Delta_i$, the least element
of $\Gamma \setminus S_i$ is also the least element of $\Gamma \setminus S_\ell$ and, since $\chi_{i K} \in \PCMinus$, we have
$\chi_{\ell K} \in \PCMinus$.  Thus
$\phi^{-1} \chi_{i K} \phi$ may be written as a product of elements in $(\PCMinus)^{\pm 1}$.
\end{proof}



Recall that the center of $W$ is the special subgroup
generated by those vertices adjacent to every other vertex.  Recall also that
Lemma \ref{WhenGAutWEqualsAutW} gives some sufficient conditions for the equality $\CAutW =\AutW$.

\begin{cor}\label{ExtensionApplication}
If $W$ is a graph product of directly-indecomposable cyclic groups and the following conditions are satisfied:
\begin{enumerate}
\item \label{TrivialCenterHypothesis} $W$ has trivial center;
\item \label{CAutWEverythingHypothesis} $\CAutW =\AutW$;
\item \label{MCSHypothesis}
 $\phi(W(\Delta_i)) = W(\Delta_i)$ for each $1 \leq i \leq N$ and each $\phi \in \QAutW$;
\end{enumerate}
then all extensions
$$1 \to W \to E \to G \to 1$$ are trivial (that is, split extensions).
\end{cor}

\begin{proof}
Conditions (\ref{CAutWEverythingHypothesis}) and (\ref{MCSHypothesis}) and
Lemma \ref{WhenAreSplittingsCompatibleLemma} give that $\AutW = \InnW \rtimes \OutW$.
Thus each homomorphism $\psi : G \to \OutW$ lifts to a homomorphism $\hat{\psi}: G \to
\AutW$ and hence determines a semidirect product $W
\rtimes_{\hat{\psi}} G$.  Condition (\ref{TrivialCenterHypothesis}) of the hypothesis ensures that there is exactly one
extension of $G$ by $W$ (up to equivalence) corresponding to any
homomorphism $\psi: G \to \OutW$ \cite[Corollary IV.6.8 p.104]{KenBrown}.
\end{proof}

In the case that $W$ is a right-angled Coxeter group, we may express the hypotheses of Corollary \ref{ExtensionApplication} entirely in terms
of the labeled-graph $\Gamma$.

\begin{cor}\label{RightAngledExtension}
If $W$ is a right-angled Coxeter group and the following conditions are satisfied:
\begin{enumerate}
\item \label{GraphCenterCondition} $\Gamma \setminus S_i \neq \emptyset$ for each $1 \leq i \leq N$;
\item \label{SymmetryCondition} $f(\Delta_i) = \Delta_i$ for each labeled-graph isomorphism $f\!:(\Gamma, \ordermap) \to (\Gamma, \ordermap)$ and each $1 \leq i \leq N$;
\item \label{StarCondition} $S_i \subseteq S_j$ if and only if $S_i = S_j$ for each $1 \leq i, j \leq N$;
\end{enumerate}
then all extensions
$$1 \to W \to E \to G \to 1$$ are trivial (that is, split extensions).
\end{cor}

\begin{proof}
Condition (\ref{GraphCenterCondition}) implies that $W$ has trivial center.  By Lemma \ref{WhenGAutWEqualsAutW}(\ref{GPPCGCondition}), $\CAutW = \AutW$.  Since $W$ is a right-angled Coxeter group, the group $\QAutW$ is generated by those automorphisms induced by symmetries of $\Gamma$ and
by automorphisms of the form $$v_i \mapsto v_i v_j, \;\;\; v_k \mapsto v_k \text{ for each } k \neq i,$$
for some $1 \leq i, j \leq N$ for which $i \neq j$ and $S_i \subseteq S_j$.
It follows that Conditions (\ref{SymmetryCondition}) and (\ref{StarCondition}) imply that
$\phi(W(\Delta_i)) = W(\Delta_i)$ for each $1 \leq i \leq N$ and each $\phi \in \QAutW$.
\end{proof}

\begin{rem}
We conjecture that, for an arbitrary graph product of directly-indecomposable cyclic groups $W$, the analogue of Corollary \ref{RightAngledExtension}
is true provided we replace Condition (\ref{StarCondition}) by the following:
\begin{enumerate}
\item [(\ref{StarCondition}')] if $S_i \subseteq S_j$ and either $\ordermap(j)$ divides $\ordermap(i)$ or $\ordermap(i) = \infty$, then $S_i = S_j$.
\end{enumerate}

\end{rem}

We now consider some applications of Theorem \ref{NiceStructureOfOutZForTree}.  Recall that we write $\vertices$ (resp. \edges) for the set of vertices
(resp. edges) of $\Gamma$ and $N = \abs{\vertices}$. Let $\leaves \subset \vertices$ denote the set of vertices
which have valence one (the `leaves' of $\Gamma$).  

\begin{cor}\label{OutWVTFandFVCD}
If $W$ is a graph product of primary cyclic groups and $\Gamma$ is a
tree, then $\OutW$ is virtually torsion-free and
$$\vcd(\OutW) = \abs{\leaves}-2.$$
\end{cor}

\begin{proof}
It follows from Lemma \ref{WhenGAutWEqualsAutW}(1) and Theorem \ref{NiceStructureOfOutZForTree} that the product
$\prod_{i=1}^N \OutZ(L_i)$ is isomorphic to a subgroup of finite index in $\OutW$. Thus it suffices to
calculate the virtual cohomological dimension of this product.

For each $i$, $W(L_i)$ is a free product of finite groups and so $\OutZ(L_i)$ is
virtually torsion-free \cite{Collins} and $\vcd(\OutZ(L_i)) =
\max\{0, \abs{L_i} - 2\}$ \cite{KrsticAndVogtmann} \cite[p.
67]{McCulloghAndMiller}.  The direct product of virtually torsion-free groups is virtually torsion-free, so
$\prod_{i=1}^N \OutZ(L_i)$ is virtually torsion-free.

It follows from \cite[Proposition VII.2.4(b) p.187]{KenBrown} (see also \cite[VII.11 exercise 2 p.229]{KenBrown}) that
the virtual cohomological dimension of a direct product is at most the sum of the virtual cohomological dimensions of the factors.  Thus we have
$$\vcd\Bigl(\prod_{i=1}^N \OutZ(L_i)\Bigr) \leq \sum_{i = 1}^{N} \max\{0, \abs{L_i} - 2\}.$$

For each $1 \leq i \leq N$, $\Out^0 \, W(L_i)$ contains a free abelian subgroup of rank $\max\{0, \abs{L_i} - 2\}$ (if $L_i = \{v_{j_1}, \dots, v_{j_M}\}$,
then $\{(\chi_{j_2\{j_3\}} \chi_{j_1\{j_3\}})$, $\dots$, $(\chi_{j_2\{j_M\}} \chi_{j_1\{j_M\}})\}$ generates a free abelian subgroup).  It follows that
the product $\prod_{i=1}^N \OutZ(L_i)$ contains a free abelian subgroup of
rank $$\sum_{i = 1}^{N} \max\{0, \abs{L_i} - 2\}$$  and
$$\vcd\Bigl(\prod_{i=1}^N \OutZ(L_i)\Bigr) = \sum_{i = 1}^{N} \max\{0, \abs{L_i} - 2\}.$$


Finally,
\begin{eqnarray*}
\sum_{i = 1}^{N} \max\{0, \abs{L_i} - 2\} & = & \biggl(\sum_{i = 1}^{N} (\abs{L_i} - 2)\biggr) + \abs{\leaves} \\
 & = & \biggl(\sum_{i = 1}^{N} \abs{L_i}\biggr) - 2 N +\abs{\leaves} \\
 & = &  2\abs{\edges} - 2 \abs{\vertices} + \abs{\leaves} \\
 & = &  2\abs{\edges} - 2 (\abs{\edges}+1) + \abs{\leaves} \\
 & = &  \abs{\leaves}-2.
\end{eqnarray*}
(the first equality holds because $\abs{L_i} - 2 < 0$ if and only if $v_i \in \leaves$ and $\abs{L_i}-2 = -1$,  the third equality holds because
each edge in $\Gamma$ contributes to $\abs{L_i}$ for two values of $i$ and the fourth equality holds because $\abs{\vertices} = \abs{\edges} + 1$).
\end{proof}

The following corollary extends the main results from \cite{MoAndSava}.

\begin{cor}\label{RLofNF}
If $W$ is a right-angled Artin group and $\Gamma$ is a
tree, then there exist regular
languages of normal forms for $\OutZ$ and $\AutZ$.
\end{cor}

\begin{proof}
Consider the structure of $\OutZ$ as described in Theorem \ref{NiceStructureOfOutZForTree}.
For each $1 \leq i \leq N$, $W(L_i)$ is a free group and there exists a regular
language of normal forms $\mathcal{N}_i$ for $\OutZ(L_i)$
\cite{MoAndSava}. Since $\Ab$ is a finitely-generated free abelian
group, there is a regular language of normal forms
$\mathcal{N}_{\Ab}$ for $\Ab$. The language $\mathcal{N}_{\Ab}
\mathcal{N}_1 \mathcal{N}_2 \dots \mathcal{N}_{N}$ is a regular
language of normal forms for $\OutZ$.

Further, $\InnW$ is a right-angled Artin group and hence is automatic \cite[Theorem B]{Meier2}.  It follows that there is a regular language of normal forms
$\mathcal{N}_{I}$ for $\InnW$.  By Theorem \ref{CSemiDirect}, the language $\mathcal{N}_{I} \mathcal{N}_{\Ab}
\mathcal{N}_1 \mathcal{N}_2 \dots \mathcal{N}_{N}$ is a regular language of normal forms for $\AutZ$.
\end{proof}

\appendix
\section{An example}\label{ExampleSection}

Let $\Gamma$ be the tree depicted in Figure \ref{TreeExampleFigure} and $\ordermap$ an order map on $\Gamma$.
\begin{figure}
\centering
\includegraphics[scale=0.3]{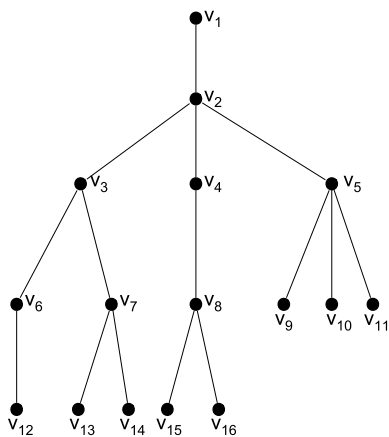}
\caption{The graph $\Gamma$ for $\S$\ref{ExampleSection}. \label{TreeExampleFigure}}
\end{figure}
By Laurence's result \cite[Theorem 4.1]{LaurenceThesis}, $\AutZ$ is generated by the set $\PC$:
\begin{eqnarray*}
\PC & = & \{\chi_{1 \{v_3, v_6, v_7, v_{12}, v_{13}, v_{14}\}}, \chi_{1 \{v_4, v_8, v_{15}, v_{16}\}}, \chi_{1 \{v_5, v_9, v_{10}, v_{11}\}},\\
    &   & \chi_{2 \{v_6, v_{12}\}}, \chi_{2 \{v_7, v_{13}, v_{14}\}}, \chi_{2 \{v_8, v_{15}, v_{16}\}}, v_{2 \{v_9\}}, v_{2 \{v_{10}\}}, v_{2 \{v_{11}\}}, \\
    &   & \chi_{3 \{v_1\}}, \chi_{3 \{v_{12}\}}, \chi_{3, \{v_{13}\}}, \chi_{3 \{v_{14}\}}, \chi_{3 \{v_4, v_8, v_{15}, v_{16}\}}, \chi_{3 \{v_5, v_9, v_{10}, v_{11}\}}, \\
    &   & \chi_{4 \{v_1\}}, \chi_{4 \{v_3, v_6, v_7, v_{12}, v_{13}, v_{14}\}}, \chi_{4 \{v_{15}\}}, \chi_{4 \{v_{16}\}}, \chi_{4 \{v_5, v_9, v_{10}, v_{11}\}}, \\
    &   & \chi_{5 \{v_1\}}, \chi_{5 \{v_3, v_6, v_7, v_{12}, v_{13}, v_{14}\}}, \chi_{5 \{v_4, v_8, v_{15}, v_{16}\}},  \\
    &   & \chi_{6 \{v_1, v_2, v_4, v_5, v_8, v_9, v_{10}, v_{11}, v_{15}, v_{16}\}}, \chi_{6 \{v_7, v_{13}, v_{14}\}},  \\
    &   & \chi_{7 \{v_1, v_2, v_4, v_5, v_8, v_9, v_{10}, v_{11}, v_{15}, v_{16}\}}, \chi_{7 \{v_6, v_{12}\}},  \\
    &   & \chi_{8 \{v_1, v_2, v_3, v_5, v_6, v_7, v_9, v_{10}, v_{11}, v_{12}, v_{13}, v_{14}\}}, \\
    &   & \chi_{9 \{v_1, v_2, v_3, v_4, v_6, v_7, v_8, v_{12}, v_{13}, v_{14}, v_{15}, v_{16}\}}, \chi_{9 \{v_{10}\}}, \chi_{9 \{v_{11}\}},   \\
    &   & \chi_{10 \{v_1, v_2, v_3, v_4, v_6, v_7, v_8, v_{12}, v_{13}, v_{14}, v_{15}, v_{16}\}}, \chi_{10 \{v_{9}\}}, \chi_{10 \{v_{11}\}},\\
    &   & \chi_{11 \{v_1, v_2, v_3, v_4, v_6, v_7, v_8, v_{12}, v_{13}, v_{14}, v_{15}, v_{16}\}}, \chi_{11 \{v_{9}\}}, \chi_{11 \{v_{10}\}},\\
    &   & \chi_{12 \{v_1, v_2, v_3, v_4, v_5, v_7, v_8, v_9, v_{10}, v_{11}, v_{13}, v_{14}, v_{15}, v_{16}\}}, \\
    &   & \chi_{13 \{v_1, v_2, v_3, v_4, v_5, v_6, v_8, v_9, v_{10}, v_{11}, v_{12}, v_{15}, v_{16}\}}, \chi_{13 \{v_{14}\}},\\
    &   & \chi_{14 \{v_1, v_2, v_3, v_4, v_5, v_6, v_8, v_9, v_{10}, v_{11}, v_{12}, v_{15}, v_{16}\}}, \chi_{14 \{v_{13}\}},\\
    &   & \chi_{15 \{v_1, v_2, v_3, v_4, v_5, v_6, v_7, v_9, v_{10}, v_{11}, v_{12}, v_{13}, v_{14}\}}, \chi_{15 \{v_{16}\}},\\
    &   & \chi_{16 \{v_1, v_2, v_3, v_4, v_5, v_6, v_7, v_9, v_{10}, v_{11}, v_{12}, v_{13}, v_{14}\}}, \chi_{16 \{v_{15}\}} \}.
\end{eqnarray*}
The sets $\PartitionSet_1, \dots, \PartitionSet_{16}$ are as follows:
\begin{eqnarray*}
\PartitionSet_1 & = & \PartitionSet_9 = \PartitionSet_{10} = \PartitionSet_{11} = \PartitionSet_{12} = \PartitionSet_{13} = \PartitionSet_{14} = \PartitionSet_{15} = \PartitionSet_{16} = \emptyset \\
\PartitionSet_2 & = & \{\chi_{1 \{v_3, v_6, v_7, v_{12}, v_{13}, v_{14}\}}, \chi_{1 \{v_4, v_8, v_{15}, v_{16}\}}, \chi_{1 \{v_5, v_9, v_{10}, v_{11}\}}, \chi_{3 \{v_1\}}, \chi_{3 \{v_4, v_8, v_{15}, v_{16}\}}, \\
                &   & \chi_{3 \{v_5, v_9, v_{10}, v_{11}\}}, \chi_{4 \{v_1\}}, \chi_{4 \{v_3, v_6, v_7, v_{12}, v_{13}, v_{14}\}}, \chi_{4 \{v_5, v_9, v_{10}, v_{11}\}}, \chi_{5 \{v_1\}}, \\
                &   & \chi_{5 \{v_3, v_6, v_7, v_{12}, v_{13}, v_{14}\}}, \chi_{5 \{v_4, v_8, v_{15}, v_{16}\}}\} \\
\PartitionSet_3 & = & \{\chi_{2 \{v_6, v_{12}\}}, \chi_{2 \{v_7, v_{13}, v_{14}\}}, \chi_{6 \{v_1, v_2, v_4, v_5, v_8, v_9, v_{10}, v_{11}, v_{15}, v_{16}\}}, \chi_{6 \{v_7, v_{13}, v_{14}\}}, \\
                &   & \chi_{7 \{v_1, v_2, v_4, v_5, v_8, v_9, v_{10}, v_{11}, v_{15}, v_{16}\}}, \chi_{7 \{v_6, v_{12}\}}\} \\
\PartitionSet_4 & = & \{\chi_{2 \{v_8, v_{15}, v_{16}\}}, \chi_{8 \{v_1, v_2, v_3, v_5, v_6, v_7, v_9, v_{10}, v_{11}, v_{12}, v_{13}, v_{14}\}}\} \\
\PartitionSet_5 & = & \{v_{2 \{v_9\}}, v_{2 \{v_{10}\}}, v_{2 \{v_{11}\}}, \chi_{9 \{v_1, v_2, v_3, v_4, v_6, v_7, v_8, v_{12}, v_{13}, v_{14}, v_{15}, v_{16}\}}, \chi_{9 \{v_{10}\}}, \chi_{9 \{v_{11}\}}, \\
                &   & \chi_{10 \{v_1, v_2, v_3, v_4, v_6, v_7, v_8, v_{12}, v_{13}, v_{14}, v_{15}, v_{16}\}}, \chi_{10 \{v_{9}\}}, \chi_{10 \{v_{11}\}}, \\
                &   & \chi_{11 \{v_1, v_2, v_3, v_4, v_6, v_7, v_8, v_{12}, v_{13}, v_{14}, v_{15}, v_{16}\}}, \chi_{11 \{v_{9}\}}, \chi_{11 \{v_{10}\}}\} \\
\PartitionSet_6 & = & \{\chi_{3 \{v_{12}\}}, \chi_{12 \{v_1, v_2, v_3, v_4, v_5, v_7, v_8, v_9, v_{10}, v_{11}, v_{13}, v_{14}, v_{15}, v_{16}\}}\} \\
\PartitionSet_7 & = & \{\chi_{3 \{v_{13}\}}, \chi_{3 \{v_{14}\}}, \chi_{13 \{v_1, v_2, v_3, v_4, v_5, v_6, v_8, v_9, v_{10}, v_{11}, v_{12}, v_{15}, v_{16}\}}, \chi_{13 \{v_{14}\}},  \\
                &   & \chi_{14 \{v_1, v_2, v_3, v_4, v_5, v_6, v_8, v_9, v_{10}, v_{11}, v_{12}, v_{15}, v_{16}\}}, \chi_{14 \{v_{13}\}}\} \\
\PartitionSet_8 & = & \{\chi_{4 \{v_{15}\}}, \chi_{4 \{v_{16}\}}, \chi_{15 \{v_1, v_2, v_3, v_4, v_5, v_6, v_7, v_9, v_{10}, v_{11}, v_{12}, v_{13}, v_{14}\}}, \chi_{15 \{v_{16}\}}, \\
                &   & \chi_{16 \{v_1, v_2, v_3, v_4, v_5, v_6, v_7, v_9, v_{10}, v_{11}, v_{12}, v_{13}, v_{14}\}}, \chi_{16 \{v_{15}\}}\}.
\end{eqnarray*}

To construct $\PCMinus$ from $\PC$, we remove the first partial conjugation from each line in the description of $\PC$ above to get:
\begin{eqnarray*}
\PCMinus & = & \{ \chi_{1 \{v_4, v_8, v_{15}, v_{16}\}}, \chi_{1 \{v_5, v_9, v_{10}, v_{11}\}}, \chi_{2 \{v_7, v_{13}, v_{14}\}},
  \chi_{2 \{v_8, v_{15}, v_{16}\}}, v_{2 \{v_9\}}, v_{2 \{v_{10}\}},  \\
    &   & v_{2 \{v_{11}\}},\chi_{3 \{v_{12}\}}, \chi_{3, \{v_{13}\}}, \chi_{3 \{v_{14}\}}, \chi_{3 \{v_4, v_8, v_{15}, v_{16}\}}, \chi_{3 \{v_5, v_9, v_{10}, v_{11}\}}, \\
    &   & \chi_{4 \{v_3, v_6, v_7, v_{12}, v_{13}, v_{14}\}}, \chi_{4 \{v_{15}\}}, \chi_{4 \{v_{16}\}}, \chi_{4 \{v_5, v_9, v_{10}, v_{11}\}},
     \chi_{5 \{v_3, v_6, v_7, v_{12}, v_{13}, v_{14}\}}, \\
    &   & \chi_{5 \{v_4, v_8, v_{15}, v_{16}\}},
    \chi_{6 \{v_7, v_{13}, v_{14}\}},
    \chi_{7 \{v_6, v_{12}\}},
    \chi_{9 \{v_{10}\}}, \chi_{9 \{v_{11}\}},
    \chi_{10 \{v_{9}\}}, \chi_{10 \{v_{11}\}},\\
    & & \chi_{11 \{v_{9}\}}, \chi_{11 \{v_{10}\}},
     \chi_{13 \{v_{14}\}},
     \chi_{14 \{v_{13}\}},
     \chi_{15 \{v_{16}\}},
     \chi_{16 \{v_{15}\}} \}.
\end{eqnarray*}


The groups $\langle \PartitionSet_1^0 \rangle , \dots, \langle \PartitionSet_{16}^0 \rangle $ are as follows:
\begin{eqnarray*}
\langle \PartitionSet_i^0 \rangle & \cong & \langle \emptyset \rangle \cong \{id\} \cong \OutW(L_i) \hbox{ for } i = 1, 9, 10, 11, 12, 13, 14, 15, 16; \\
  & & \\
\langle \PartitionSet_2^0 \rangle & \cong & \langle \{\chi_{1 \{v_4, v_8, v_{15}, v_{16}\}}, \chi_{1 \{v_5, v_9, v_{10}, v_{11}\}}, \chi_{3 \{v_4, v_8, v_{15}, v_{16}\}}, \chi_{3 \{v_5, v_9, v_{10}, v_{11}\}}, \\
  & &  \chi_{4 \{v_3, v_6, v_7, v_{12}, v_{13}, v_{14}\}}, \chi_{4 \{v_5, v_9, v_{10}, v_{11}\}}, \chi_{5 \{v_3, v_6, v_7, v_{12}, v_{13}, v_{14}\}}, \chi_{5 \{v_4, v_8, v_{15}, v_{16}\}}\} \rangle \\
  & \cong & \OutZ(L_2); \\
  & & \\
\end{eqnarray*}
\begin{eqnarray*}
\langle \PartitionSet_3^0 \rangle & \cong & \langle \{\chi_{2 \{v_7, v_{13}, v_{14}\}}, \chi_{6 \{v_7, v_{13}, v_{14}\}}, \chi_{7 \{v_6, v_{12}\}}\} \rangle \\
  & \cong & \OutZ(L_3); \\
  & & \\
\langle \PartitionSet_4^0 \rangle & \cong & \langle \{\chi_{2 \{v_8, v_{15}, v_{16}\}}\} \rangle \\
  & \cong & \Integer_{\ordermap(2)} \\
  & \cong & \Integer_{\ordermap(2)} \times \OutZ(L_4); \\
  & & \\
\langle \PartitionSet_5^0 \rangle & \cong & \langle \{v_{2 \{v_9\}}, v_{2 \{v_{10}\}}, v_{2 \{v_{11}\}}, \chi_{9 \{v_{10}\}}, \chi_{9 \{v_{11}\}}, \chi_{10 \{v_{9}\}}, \chi_{10 \{v_{11}\}}, \chi_{11 \{v_{9}\}}, \chi_{11 \{v_{10}\}}\} \rangle \\
  & \cong & \langle \{v_{2 \{v_9\}}\} \cup \{v_{2 \{v_{10}\}}, v_{2 \{v_{11}\}}, \chi_{9 \{v_{10}\}}, \chi_{9 \{v_{11}\}}, \chi_{10 \{v_{9}\}}, \chi_{10 \{v_{11}\}}, \chi_{11 \{v_{9}\}}, \chi_{11 \{v_{10}\}}\} \rangle  \\
                                  & \cong & \langle \{v_{2 \{v_9, v_{10}, v_{11}\}}\} \cup \{v_{2 \{v_{10}\}}, v_{2 \{v_{11}\}}, \chi_{9 \{v_{10}\}}, \chi_{9 \{v_{11}\}}, \chi_{10 \{v_{9}\}}, \chi_{10 \{v_{11}\}}, \chi_{11 \{v_{9}\}}, \chi_{11 \{v_{10}\}}\} \rangle  \\
                                  & \cong & \Integer_{\ordermap(2)} \times \OutZ(L_5);\\
                                  & & \\
\langle \PartitionSet_6^0 \rangle & \cong & \langle \{\chi_{3 \{v_{12}\}}\} \rangle \\
  & \cong & \Integer_{\ordermap(3)} \\
  & \cong & \Integer_{\ordermap(3)} \times \OutZ(L_6); \\
  & & \\
\langle \PartitionSet_7^0 \rangle & \cong & \langle \{\chi_{3 \{v_{13}\}}, \chi_{3 \{v_{14}\}}, \chi_{13 \{v_{14}\}},  \chi_{14 \{v_{13}\}}\} \rangle \\
                                  & \cong & \langle \{\chi_{3 \{v_{13}\}}\} \cup \{ \chi_{3 \{v_{14}\}}, \chi_{13 \{v_{14}\}},  \chi_{14 \{v_{13}\}}\} \rangle  \\
                                  & \cong & \langle \{\chi_{3 \{v_{13}, v_{14}\}}\} \cup \{ \chi_{3 \{v_{14}\}}, \chi_{13 \{v_{14}\}},  \chi_{14 \{v_{13}\}}\} \rangle  \\
                                  & \cong & \Integer_{\ordermap(3)} \times \OutZ(L_7);\\
                                  & & \\
\langle \PartitionSet_8^0 \rangle & \cong & \langle \{\chi_{4 \{v_{15}\}}, \chi_{4 \{v_{16}\}}, \chi_{15 \{v_{16}\}}, \chi_{16 \{v_{15}\}}\} \rangle \\
                                  & \cong & \langle \{\chi_{4 \{v_{15}\}}\} \cup \{\chi_{4 \{v_{16}\}}, \chi_{15 \{v_{16}\}}, \chi_{16 \{v_{15}\}}\} \rangle  \\
                                  & \cong & \langle \{\chi_{4 \{v_{15}, v_{16}\}}\} \cup \{\chi_{4 \{v_{16}\}}, \chi_{15 \{v_{16}\}}, \chi_{16 \{v_{15}\}}\} \rangle  \\
                                  & \cong & \Integer_{\ordermap(4)} \times \OutZ(L_8).\\
\end{eqnarray*}



\bibliographystyle{abbrv}
\bibliography{AutomorphismsOfGPAGroupsBib}

\end{document}